\documentclass[11pt]{amsart}

\usepackage{anysize} \marginsize{1.3in}{1.3in}{1in}{1in}

\usepackage{xcolor}
\usepackage{amsmath}
\usepackage{mathtools}
\usepackage[all]{xy}
\usepackage[utf8]{inputenc}
\usepackage{varioref}
\usepackage{amsfonts}
\usepackage{amssymb}
\usepackage{bbm}
\usepackage{xparse}
\usepackage{graphicx}
\usepackage{mathrsfs}
\usepackage[hypertexnames=false,backref=page,pdftex,
 	pdfpagemode=UseNone,
 	breaklinks=true,
 	extension=pdf,
 	colorlinks=true,
 	linkcolor=blue,
 	citecolor=red,
 	urlcolor=blue,
 ]{hyperref}

\let\oldeqref\eqref
\makeatletter
\RenewDocumentCommand\eqref{s m}{%
  \IfBooleanTF#1%
  {\textup{\tagform@{\ref*{#2}}}}% If a star is seen
  {\oldeqref{#2}}%                 If no star is seen
}
\makeatother

\newcommand{\sA}{{\mathcal A}}

\newcommand{\sC}{{\mathcal C}}
\newcommand{\sD}{{\mathcal D}}

\newcommand{\sK}{{\mathcal K}}

\newcommand{\sO}{{\mathcal O}}

\newcommand{\sU}{{\mathcal U}}
\newcommand{\sV}{{\mathcal V}}

\newcommand{\sX}{{\mathcal X}}
\newcommand{\sY}{{\mathcal Y}}

\newcommand{\scrC}{{\mathscr C}}

\newcommand{\scrL}{{\mathscr L}}

\newcommand{\scrU}{{\mathscr U}}

\newcommand{\scrX}{{\mathscr X}}
\newcommand{\scrY}{{\mathscr Y}}
\newcommand{\scrZ}{{\mathscr Z}}

\newcommand{\C}{{\mathbb C}}

\newcommand{\N}{{\mathbb N}}
\renewcommand{\P}{{\mathbb P}}
\newcommand{\Q}{{\mathbb Q}}
\newcommand{\R}{{\mathbb R}}

\newcommand{\Z}{{\mathbb Z}}

\newcommand{\gothM}{{\mathfrak M}}

\newcommand{\abs}[1]{{\left|#1\right|}}

\newcommand{\Aut}{\operatorname{Aut}}

\newcommand{\ba}[1]{\overline{#1}}
\newcommand{\BM}{{\operatorname{BM}}}

\newcommand{\ch}{{\rm ch}}

\newcommand{\chern}{{\rm c}}

\newcommand{\codim}{\operatorname{codim}}

\newcommand{\Def}{{\operatorname{Def}}}

\newcommand{\del}{\partial}

\newcommand{\di}{\partial}
\newcommand{\dibar}{\ba\partial}

\newcommand{\End}{{\operatorname{End}}}
\newcommand{\Ext}{{\operatorname{Ext}}}
\newcommand{\sExt}{{\mathcal{E}xt}}

\newcommand{\Hdg}{\operatorname{Hdg}}
\newcommand{\Hom}{\operatorname{Hom}}

\renewcommand{\implies}{\Rightarrow}

\newcommand{\id}{{\rm id}}
\newcommand{\img}{\operatorname{im}}
\newcommand{\into}{{\, \hookrightarrow\,}}
\newcommand{\isom}{{\ \cong\ }}

\newcommand{\lt}{{\rm{lt}}}

\renewcommand{\O}{{\rm O}}

\newcommand{\ohnenull}{{\ \setminus \{0\} }}

\newcommand{\PGL}{{\rm PGL}}
\newcommand{\Pic}{\operatorname{Pic}}

\newcommand{\ratl}{\dashrightarrow}

\newcommand{\red}{{\operatorname{red}}}
\newcommand{\reg}{{\operatorname{reg}}}

\newcommand{\rk}{{\rm rk}}

\newcommand{\sing}{{\operatorname{sing}}}

\newcommand{\Spec}{\operatorname{Spec}}

\newcommand{\SO}{{\rm SO}}

\newcommand{\Sym}{{\rm Sym}}

\renewcommand{\to}[1][]{\xrightarrow{\ #1\ }}

\newcommand{\tensor}{\otimes}

\newcommand{\veps}{\varepsilon}
\newcommand{\vphi}{\varphi}

\newcommand{\wt}[1]{{\widetilde{#1}}}

\makeatletter
\newcommand*{\da@rightarrow}{\mathchar"0\hexnumber@\symAMSa 4B }
\newcommand*{\da@leftarrow}{\mathchar"0\hexnumber@\symAMSa 4C }
\newcommand*{\xdashrightarrow}[2][]{%
  \mathrel{%
    \mathpalette{\da@xarrow{#1}{#2}{}\da@rightarrow{\,}{}}{}%
  }%
}
\newcommand{\xdashleftarrow}[2][]{%
  \mathrel{%
    \mathpalette{\da@xarrow{#1}{#2}\da@leftarrow{}{}{\,}}{}%
  }%
}
\newcommand*{\da@xarrow}[7]{%
  \sbox0{$\ifx#7\scriptstyle\scriptscriptstyle\else\scriptstyle\fi#5#1#6\m@th$}%
  \sbox2{$\ifx#7\scriptstyle\scriptscriptstyle\else\scriptstyle\fi#5#2#6\m@th$}%
  \sbox4{$#7\dabar@\m@th$}%
  \dimen@=\wd0 %
  \ifdim\wd2 >\dimen@
    \dimen@=\wd2 %   
  \fi
  \count@=2 %
  \def\da@bars{\dabar@\dabar@}%
  \@whiledim\count@\wd4<\dimen@\do{%
    \advance\count@\@ne
    \expandafter\def\expandafter\da@bars\expandafter{%
      \da@bars
      \dabar@ 
    }%
  }%  
  \mathrel{#3}%
  \mathrel{%   
    \mathop{\da@bars}\limits
    \ifx\\#1\\%
    \else
      _{\copy0}%
    \fi
    \ifx\\#2\\%
    \else
      ^{\copy2}%
    \fi
  }%   
  \mathrel{#4}%
}
\makeatother

\newtheoremstyle{citing}% name
  {}%      Space above, empty = `usual value'
  {}%      Space below
  {\itshape}% Body font
  {}%         Indent amount (empty = no indent, \parindent = para indent)
  {\bfseries}% Thm head font
  {\textbf{.}}%        Punctuation after thm head
  {.5em}%     Space after thm head: " " = normal interword space;
  {\thmnote{#3}}% Thm head spec

\theoremstyle{plain}

\newtheorem{theorem}[subsection]{Theorem}

\newtheorem{lemma}[subsection]{Lemma}
\newtheorem{corollary}[subsection]{Corollary}

\newtheorem{proposition}[subsection]{Proposition}

\theoremstyle{remark}
\newtheorem{example}[subsection]{Example}

\theoremstyle{definition}

\newtheorem{definition}[subsection]{Definition}

\numberwithin{equation}{section}

\theoremstyle{remark}
\newtheorem{remark}[subsection]{Remark}
\newtheorem*{claim}{Claim}

{\theoremstyle{citing}
}

\makeatletter
\newsavebox\myboxA
\newsavebox\myboxB
\newlength\mylenA

\newcommand*\xtilde[2][0.8]{%
    \sbox{\myboxA}{$\m@th#2$}%
    \setbox\myboxB\null% Phantom box
    \ht\myboxB=\ht\myboxA%
    \dp\myboxB=\dp\myboxA%
    \wd\myboxB=#1\wd\myboxA% Scale phantom
    \sbox\myboxB{$\m@th\widetilde{\copy\myboxB}$}%  Overlined phantom
    \setlength\mylenA{\the\wd\myboxA}%   calc width diff
    \addtolength\mylenA{-\the\wd\myboxB}%
    \ifdim\wd\myboxB<\wd\myboxA%
       \rlap{\hskip 0.5\mylenA\usebox\myboxB}{\usebox\myboxA}%
    \else
        \hskip -0.5\mylenA\rlap{\usebox\myboxA}{\hskip 0.5\mylenA\usebox\myboxB}%
    \fi}
\newbox\usefulbox

\def\getslant #1{\strip@pt\fontdimen1 #1}

\def\xxtilde #1{\mathchoice
 {{\setbox\usefulbox=\hbox{$\m@th\displaystyle #1$}%
    \dimen@ \getslant\the\textfont\symletters \ht\usefulbox
    \divide\dimen@ \tw@ 
    \kern\dimen@ 
    \xtilde{\kern-\dimen@ \box\usefulbox\kern\dimen@ }\kern-\dimen@ }}
 {{\setbox\usefulbox=\hbox{$\m@th\textstyle #1$}%
    \dimen@ \getslant\the\textfont\symletters \ht\usefulbox
    \divide\dimen@ \tw@ 
    \kern\dimen@ 
    \xtilde{\kern-\dimen@ \box\usefulbox\kern\dimen@ }\kern-\dimen@ }}
 {{\setbox\usefulbox=\hbox{$\m@th\scriptstyle #1$}%
    \dimen@ \getslant\the\scriptfont\symletters \ht\usefulbox
    \divide\dimen@ \tw@ 
    \kern\dimen@ 
    \xtilde{\kern-\dimen@ \box\usefulbox\kern\dimen@ }\kern-\dimen@ }}
 {{\setbox\usefulbox=\hbox{$\m@th\scriptscriptstyle #1$}%
    \dimen@ \getslant\the\scriptscriptfont\symletters \ht\usefulbox
    \divide\dimen@ \tw@ 
    \kern\dimen@ 
    \xtilde{\kern-\dimen@ \box\usefulbox\kern\dimen@ }\kern-\dimen@ }}%
 {}}
\newcommand*\xoverline[2][0.75]{%
    \sbox{\myboxA}{$\m@th#2$}%
    \setbox\myboxB\null% Phantom box
    \ht\myboxB=\ht\myboxA%
    \dp\myboxB=\dp\myboxA%
    \wd\myboxB=#1\wd\myboxA% Scale phantom
    \sbox\myboxB{$\m@th\overline{\copy\myboxB}$}%  Overlined phantom
    \setlength\mylenA{\the\wd\myboxA}%   calc width diff
    \addtolength\mylenA{-\the\wd\myboxB}%
    \ifdim\wd\myboxB<\wd\myboxA%
       \rlap{\hskip 0.5\mylenA\usebox\myboxB}{\usebox\myboxA}%
    \else
        \hskip -0.5\mylenA\rlap{\usebox\myboxA}{\hskip 0.5\mylenA\usebox\myboxB}%
    \fi}
\def\xxoverline #1{\mathchoice
 {{\setbox\usefulbox=\hbox{$\m@th\displaystyle #1$}%
    \dimen@ \getslant\the\textfont\symletters \ht\usefulbox
    \divide\dimen@ \tw@ 
    \kern\dimen@ 
    \overline{\kern-\dimen@ \box\usefulbox\kern\dimen@ }\kern-\dimen@ }}
 {{\setbox\usefulbox=\hbox{$\m@th\textstyle #1$}%
    \dimen@ \getslant\the\textfont\symletters \ht\usefulbox
    \divide\dimen@ \tw@ 
    \kern\dimen@ 
    \xoverline{\kern-\dimen@ \box\usefulbox\kern\dimen@ }\kern-\dimen@ }}
 {{\setbox\usefulbox=\hbox{$\m@th\scriptstyle #1$}%
    \dimen@ \getslant\the\scriptfont\symletters \ht\usefulbox
    \divide\dimen@ \tw@ 
    \kern\dimen@ 
    \xoverline{\kern-\dimen@ \box\usefulbox\kern\dimen@ }\kern-\dimen@ }}
 {{\setbox\usefulbox=\hbox{$\m@th\scriptscriptstyle #1$}%
    \dimen@ \getslant\the\scriptscriptfont\symletters \ht\usefulbox
    \divide\dimen@ \tw@ 
    \kern\dimen@ 
    \xoverline{\kern-\dimen@ \box\usefulbox\kern\dimen@ }\kern-\dimen@ }}%
 {}}
\makeatother

\newcommand{\PH}{{\operatorname{PH}}}

\renewcommand{\div}{{\operatorname{div}}}
\newcommand{\rrk}{{\operatorname{rrk}}}
\newcommand{\ddbar}{{\del\bar\del}}

\def\Mon{\operatorname{Mon}}

\newcommand{\Exc}{{\operatorname{Exc}}}
\newcommand{\Gr}{{\operatorname{Gr}}}
\newcommand{\DP}{{\operatorname{DP}}}
\newcommand{\tf}{\mathrm{tf}}
\def\sp{\operatorname{sp}}

{\theoremstyle{definition}
\newtheorem{Ex}[subsection]{Example}%[section]}

\renewcommand{\sU}{\scrU}
\renewcommand{\sX}{\scrX}
\renewcommand{\sY}{\scrY}

\makeatletter
\newcommand{\mylabel}[2]{#2\def\@currentlabel{#2}\label{#1}}
\makeatother

\makeatletter
\newcommand{\Mac}{}
\DeclareRobustCommand{\Mac}{%
  M%
  \raisebox{\dimexpr\fontcharht\font`M-\height}{%
    \check@mathfonts\fontsize{\sf@size}{0}\selectfont
    c%
  }%
}
\makeatother

\title{The global moduli theory of symplectic varieties}

\author{Benjamin Bakker}
\address{Benjamin Bakker\\ Department of Mathematics\\ University of Illinois at Chicago\\ 851 S. Morgan St., Chicago, IL 60607}
\email{bakker.uic@gmail.com}

\author{Christian Lehn}
\address{Christian Lehn\\ Fakult\"at f\"ur Mathematik\\ Technische Universit\"at Chemnitz\\
Reichenhainer Stra\ss e 39, 09126 Chemnitz, Germany}
\email{christian.lehn@mathematik.tu-chemnitz.de}

\let\origmaketitle\maketitle
\def\maketitle{
  \begingroup
  \def\uppercasenonmath##1{} % this disables uppercasing title
  \let\MakeUppercase\relax % this disables uppercasing authors
  \origmaketitle
  \endgroup
}

\begin{document}
\thispagestyle{empty}

\begin{abstract}
We develop the global moduli theory of symplectic varieties {in the sense of Beauville}.  We prove a number of analogs of classical results from the smooth case, including a global Torelli theorem.  In particular, this yields a new proof of Verbitsky's global Torelli theorem in the smooth case (assuming $b_2\geq 5$) which does not use the existence of a hyperk\"ahler metric or twistor deformations.
\end{abstract}


\maketitle

\setlength{\parindent}{1em}
\setcounter{tocdepth}{1}

% 14B07   	Deformations of singularities
% 32G13   	Analytic moduli problems
% 14J10   	Families, moduli, classification: algebraic theory
% 32J27   	Compact Kähler manifolds: generalizations, classification
% 32S15   	Equisingularity (topological and analytic)
% 32S45   	Modifications; resolution of singularities
% 53C26   	Hyper-Kähler and quaternionic Kähler geometry, "special'' geometry

\tableofcontents

%------------------------------------------------------------------------------------------
\section{Introduction}\label{section intro}
%------------------------------------------------------------------------------------------
\thispagestyle{empty}

A \emph{symplectic variety} $X$ (in the sense of Beauville \cite{Bea00}) is a normal variety admitting a nondegenerate closed holomorphic $2$-form $\sigma\in H^0(X^\reg,\Omega^2_{X^\reg})$ on its regular part which extends holomorphically on some resolution of singularities $\pi:Y\to X$.  If $X$ is compact, $H^1(X,\sO_X)=0$, and $\sigma$ is unique up to scaling, we say $X$ is a \emph{primitive} symplectic variety.  We consider these varieties a singular analog of (compact) irreducible symplectic manifolds which is as general as possible such that a reasonable global moduli theory can still be established.
 
Irreducible symplectic manifolds are one of the three main building blocks of compact Kähler manifolds with vanishing first Chern class by a theorem of Beauville--Bogomolov \cite[Théorème 1]{Bea83}, and their geometry is very rich.  In particular, Verbitsky's global Torelli theorem \cite[Theorem 1.17]{Ver13} gives a precise description of the global deformations of a symplectic manifold in terms of the Hodge structure on its second cohomology. 

Recent work of Druel--Greb--Guenancia--Höring--Kebekus--Peternell \cite{GKKP11,DG18, Dru18, Gue16, GGK19, HP19} has shown a version of the above Beauville--Bogomolov decomposition theorem for singular projective varieties with trivial canonical class, see \cite[Theorem 1.5]{HP19}, and the ``holomorphic-symplectic" factors\footnote{They were called \emph{irreducible symplectic} by Greb--Kebekus--Peternell \cite[Definition 8.16]{GKP16}, where \emph{irreducible} refers to the decomposition theorem.} that show up are a special case of the primitive symplectic varieties we consider.  This level of generality is important because singularities are often unavoidable in higher-dimensional geometry, for instance in the minimal model program.  Our results show that the geometry of singular holomorphic-symplectic varieties enjoys the same richness as that of smooth ones, and deformation theory---especially deformations to \emph{non-projective} varieties---is as essential a part of the picture as in the smooth case. Interestingly, whereas it has proven difficult to produce new deformation types of \emph{smooth} irreducible symplectic varieties, in the singular case a number of ``new" deformation types---that is, deformation types which do not seemingly arise from holomorphic symplectic manifolds---can be constructed, see Example \ref{exhs}(2).

Our main result is a global Torelli theorem for primitive symplectic varieties in general with surjectivity of the period map in the $\Q$-factorial\footnote{There is a subtlety with the definition of $\Q$-factoriality in the analytic category:  requiring every divisor to be $\Q$-Cartier is potentially different from requiring every rank one torsion-free sheaf to have an invertible reflexive power (see Section \ref{paragraph qfact}).}  terminal case.  Before stating the theorem, let us fix some notation.  The torsion-free part $H^2(X,\Z)_\tf:=H^2(X,\Z)/\mathrm{torsion}$ of the second cohomology of a primitive symplectic variety $X$ carries a pure weight two Hodge structure (see Lemma \ref{lemma peternell}) which is further endowed with an integral locally trivial deformation-invariant quadratic form $q_{X}$ called the Beauville--Bogomolov--Fujiki (BBF) form (see Section~\ref{section bbf details}).  Fixing a lattice $\Lambda$ and denoting its quadratic form by $q$, a $\Lambda$-marking of $X$ is an isomorphism $\mu: (H^2(X,\Z)_\tf,q_{X})\xrightarrow{\cong} (\Lambda,q)$.  The set of isomorphism classes of $\Lambda$-marked primitive symplectic varieties $(X,\mu)$ is given the structure of an analytic space $\gothM_\Lambda$ by gluing the bases of \emph{locally trivial} Kuranishi families (see Definition~\ref{definition marked moduli}).  In fact, $\gothM_\Lambda$ is a not-necessarily-Hausdorff complex manifold by the unobstructedness of locally trivial deformations (see Theorem~\ref{theorem deflt is smooth}).  

We obtain a period map $P:\mathfrak{M}_\Lambda\to\Omega_\Lambda$ to the period domain $\Omega_\Lambda\subset \P( \Lambda_\C)$ by sending $(X,\mu)$ to $\mu(H^{2,0}(X))$ and it is a local isomorphism (see Proposition \ref{proposition local torelli}). There is a Hausdorff reduction $H:\gothM_\Lambda\to\overline\gothM_\Lambda$ where $\overline\gothM_\Lambda$ is a Hausdorff complex manifold and $H$ identifies inseparable points (see Section \ref{section torelli}), and we moreover have a factorization
\[\xymatrix{
&\overline{\mathfrak{M}}_\Lambda\ar[rd]^{\xxoverline{P}}&\\
\mathfrak{M}_\Lambda\ar[ru]^{H}\ar[rr]_{P}&&\Omega_\Lambda.
}\]

We now state our main result:

\begin{theorem}\label{maintorelli} Assume that $\rk(\Lambda)\geq 5$.  Then for each connected component $\gothM$ of the $\Lambda$-marked moduli space $\gothM_\Lambda$ we have:
\begin{enumerate}
\item The monodromy group $\Mon(\gothM)\subset \O(\Lambda)$ is of finite index;
\item $P:\mathfrak{M}\to \Omega_\Lambda$ is bijective over Mumford--Tate general points and in general the fibers consist of pairwise bimeromorphic varieties;
\item $\xxoverline{P}:\overline{\mathfrak{M}}\to \Omega_\Lambda$ is an isomorphism onto the complement of countably many maximal Picard rank periods;
\item If moreover one point of $\gothM$ corresponds to a primitive symplectic variety with $\Q$-factorial terminal singularities, then the same is true of every point and $\xxoverline{P}:\overline{\mathfrak{M}}\to \Omega_\Lambda$ is an isomorphism.

\end{enumerate}

\end{theorem}

Theorem \ref{maintorelli} of course also applies to the smooth case, and yields a new proof of Verbitsky's global Torelli theorem. Note that $\Q$-factorial terminal singularities form a natural class of singularities for symplectic varieties---see Example \ref{exhs} for some examples.  First, such singularities are well-suited to MMP techniques (see for example Section \ref{section proj degen} and Example \ref{exhs}(\ref{egsacca})).  Second,  symplectic varieties always have canonical singularities, and any projective primitive symplectic variety $X$ admits a crepant partial resolution with $\Q$-factorial terminal singularities $X'$ (a so-called $\Q$-factorial terminalization) whose deformation theory controls that of $X$ (see Section \ref{sect Qfact}).  Note that a $\Q$-factorial terminal $K$-trivial variety does not admit a further crepant resolution.  A version of Theorem \ref{maintorelli} has been proven by Menet \cite{Men20} for a symplectic varieties with quotient singularities; see Example~\ref{exhs}(\ref{egsingmod}) for an explicit example of a $\Q$-factorial terminal symplectic variety which does not have quotient singularities.

In \cite[Theorem 1.3]{BL16} the authors prove Theorem \ref{maintorelli} (\emph{with} surjectivity in part (3)) in the case where $\gothM$ parametrizes primitive symplectic varieties admitting a crepant resolution. The proof crucially uses that simultaneous crepant resolutions exist in locally trivial families of such varieties, as then Verbitsky's global Torelli theorem can be applied to the crepant resolution. Note that by definition, $\gothM$ consists of varieties of a fixed locally trivial deformation type which allows one to prove that either all varieties it parametrizes admit a crepant resolution or none.

The main difficulty in the general setting is that while one could try to reduce to the $\Q$-factorial terminal case by passing to a simultaneous $\Q$-factorial terminalization, even in this case a new strategy is needed as Verbitsky's proof (as well as Huybrechts' proof of the surjectivity of the period map \cite[Theorem 8.1]{Huy99}) fundamentally uses the existence of hyperk\"ahler metrics and twistor deformations.  We instead prove Theorem~\ref{maintorelli} directly using global results on the geometry of the period domain via Ratner theory (as first investigated by Verbitsky \cite{Ver15,Ver17}) together with finiteness results coming from algebraic geometry.  The surjectivity in Theorem \ref{maintorelli} then follows from a generalization to the $\Q$-factorial terminal case of work of Koll\'ar--Laza--Sacc\`a--Voisin \cite{KLSV18} on projective degenerations using MMP techniques.

In fact, there is another problem with the naive generalization of the argument of \cite{BL16}: $\Q$-factorial terminalizations are not guaranteed to exist in the analytic setting.  In the projective case the existence of a $\Q$-factorial terminalization is a consequence of deep results of Birkar--Cascini--Hacon--\Mac Kernan \cite{BCHM10} on the termination of an appropriate version of the MMP, but it is not even clear \emph{a priori} that a symplectic variety can be deformed to a projective one (although Namikawa \cite{Nam02} has results in this direction).  For this reason, we need a projectivity criterion for symplectic varieties, analogous to Huybrechts' criterion \cite[Theorem~3.11]{Huy99} for hyperk\"ahler manifolds:
\begin{theorem}\label{mainproj} Let $X$ be a primitive symplectic variety, and assume $\alpha \in H^2(X,\Q)$ is a $(1,1)$-class with $q_X(\alpha)> 0$.  Then $X$ is projective.
\end{theorem}
\begin{corollary}Every primitive symplectic variety is locally trivially deformation equivalent to a projective primitive symplectic variety.
\end{corollary}
The proof uses a (weak) singular analog of the Demailly--P\u{a}un theorem on the numerical characterization of the K\"ahler cone.  

As an application of Theorem \ref{maintorelli}, we can in fact conclude that terminalizations of symplectic varieties exist in the \emph{non-projective} case, up to a bimeromorphism:

\begin{theorem}\label{mainterminal}  Let $X$ be a primitive symplectic variety with $b_2(X)\geq 5$.  Then there is a primitive sympletic variety $X'$ that is bimeromorphic and locally trivially deformation-equivalent to $X$ that admits a $\Q$-factorial terminalization:  that is, there exists a (compact) $\Q$-factorial terminal K\"ahler variety $Y$ and a crepant map $\pi:Y\to X'$.  
\end{theorem}

We view Theorem \ref{mainterminal} as an indication that the deformation theoretic tools we develop might be used to generalize the MMP for projective symplectic varieties \cite{Dru11,LP16} to the K\"ahler setting, and this will be pursued in a subsequent paper. 

In addition to the global arguments, the proofs of Theorems \ref{maintorelli}, \ref{mainproj}, and \ref{mainterminal} require a careful analysis of the infinitesimal locally trivial deformation theory of not-necessarily-projective symplectic varieties.  There are a number of new complications all critically stemming from the fact that one can no longer bootstrap classical results on the geometry of hyperk\"ahler manifolds via passing to a crepant resolution.  In particular, we must provide:
\begin{enumerate}
\item[(i)] An analysis of the Hodge theory of rational and symplectic singularities in the non-projective setting, using recent results of Kebekus--Schnell \cite{KS18} on extending holomorphic forms.
\item[(ii)] An adaptation of the results of Koll\'ar--Laza--Sacc\`a--Voisin \cite{KLSV18} on limits of projective families in the singular setting.  This requires a singular analog of a theorem of Verbitsky saying that for a primitive symplectic variety $X$, the cup product map $\Sym^kH^2(X,\Q)\to H^{2k}(X,\Q)$ is injective for $2k\leq\dim X$.
\item[(iii)] A description of the deformation theory of terminalizations.  In particular, this requires a careful treatment of $\Q$-factoriality in the analytic category, as there are several nonequivalent generalizations of the corresponding notion in the algebraic category.
\end{enumerate}

\subsection*{Previous work}
In \cite{BL16} the authors extended many of the classical results about compact irreducible symplectic manifolds to primitive symplectic varieties admitting a crepant resolution through the study of their locally trivial deformations.  Menet \cite{Men20} has proven a version of the global Torelli theorem for certain primitive symplectic varieties with orbifold singularities using twistor deformations. There are many interesting ideas in his work that have influenced parts of the present paper, especially concerning the projectivity criterion.  The local deformation theory (and in particular the local Torelli theorem) of primitive symplectic varieties has been treated by many authors, notably by Namikawa \cite{Nam01a,Nam01b,Nam06} and Kirschner  \cite{Kir15}.

\subsection*{Outline}
In Section~\ref{section prelim} we review basic notions and results about the Hodge theory of rational singularities, Kähler spaces, big and nef classes, and $\Q$-factoriality in the analytic category. Section  \ref{section symplectic} is devoted to primitive symplectic varieties and their Hodge theory.  In Section~\ref{section defo} we show locally trivial deformations of symplectic varieties are unobstructed.  In Section~\ref{section BBF} we recall the BBF form and deduce the local Torelli theorem.  We also analyze the deformation theory of $\Q$-factorial terminalizations and prove some topological results, including the existence of Fujiki relations and the analog of a theorem of Verbitsky  discussed in (ii) above.  In Section~\ref{section projectivity criterion} we prove a (weak) singular analog of the Demailly--P\u{a}un theorem and apply it to deduce the projectivity criterion, Theorem \ref{mainproj} (see Theorem \ref{theorem projectivity criterion}).  We also prove analogs of results of Huybrechts \cite{Huy99} and \cite{BL16} on the inseparability of bimeromorphic symplectic varieties in moduli, including part (2) of Theorem \ref{maintorelli} (see Theorem \ref{theorem inseparable implies birational}  and Corollary \ref{corollary huybrechts strong}).  In Section~\ref{section proj degen} we indicate the necessary changes to \cite{KLSV18} to show the existence of limits of projective families for which the period does not degenerate in the $\Q$-factorial terminal setting.  In Section~\ref{section torelli} we prove parts (1), (3), and (4) of Theorem \ref{maintorelli} (see Theorem \ref{theorem torelli}).  In Section~\ref{section qfactorial terminalizations} we apply the deformation theory of terminalizations and the global Torelli theorem to prove Theorem \ref{mainterminal} (see Theorem \ref{theorem qfactorial terminalization}).

For those interested in the proof of the global Torelli theorem in the smooth case, Section~\ref{section torelli} can be read independently, as the results used from previous sections are standard in the smooth case\footnote{Except for the required results from \cite{KLSV18}, which can be quoted without modification.}. 

\subsection*{Acknowledgments.} We benefited from discussions, remarks, emails of Valery Alexeev, Andreas Höring, Daniel Huybrechts, Stefan Kebekus, Manfred Lehn, Thomas Peternell, Antonio Rapagnetta, Bernd Schober, and Christian Schnell.  The first named author would like to thank Giulia Sacc\`a for conversations related to Section~\ref{section proj degen}. Both authors are grateful to the referees for a very careful reading and many suggestions that have greatly improved the article.

Benjamin Bakker was partially supported by NSF grants DMS-1702149 and DMS-1848049. Christian Lehn was supported by the DFG through the research grants Le 3093/2-1, Le 3093/2-2, and  Le 3093/3-1.

\subsection*{Notation and Conventions} 
A resolution of singularities of a variety $X$ is a proper surjective bimeromorphic morphism $\pi:Y \to X$ from a nonsingular variety $Y$. The term variety will denote an integral separated scheme of finite type over $\C$ in the algebraic setting or an irreducible and reduced separated complex space in the complex analytic setting.

%------------------------------------------------------------------------------------------
\section{Preliminaries}\label{section prelim}
%------------------------------------------------------------------------------------------

A complex variety $X$ is said to have rational singularities if it is normal and for any resolution of singularities $\pi:Y \to X$ and any $i > 0$ one has $R^i\pi_*\sO_Y = 0$.  Recall that the Fujiki class $\scrC$ consists of all those compact complex varieties which are meromorphically dominated by a compact Kähler manifold, see \cite[Definition~1.1]{Fuj78}. This is equivalent to saying that there is a resolution of singularities by a compact Kähler manifold by \cite[Lemma~1.1]{Fuj78}.

The following lemma is well-known; we refer to \cite[Lemma 2.1]{BL16} for a proof. 

\begin{lemma}\label{lemma peternell}
Let $\pi:Y\to X$ be a proper bimeromorphic morphism where $X$ is a complex variety with rational singularities. Then, $\pi^*:H^1(X,\Z) \to H^1(Y,\Z)$ is an isomorphism and the sequence
\begin{equation}\label{eq leray}
 0 \to H^2(X,\Z) \to[\pi^*] H^2(Y,\Z) \to H^0(X,R^2\pi_*\Z)
\end{equation}
is exact. In particular, if $X$ is compact and $Y$ is a compact manifold of Fujiki class $\scrC$, then $H^i(X,\Z)$ carries a pure Hodge structure for $i=1, 2$. Moreover, $\pi^*H^{1,1}(X,\Z)$ is the subspace of $H^{1,1}(Y,\Z)$ of all classes that vanish on the classes of $\pi$-exceptional curves.
\end{lemma} 

%------------------------------------------------------------------------------------------
For a complex space $X$, recall that $\Omega_X^{[p]}$ denotes the sheaf of reflexive $p$-forms:
\begin{definition}\label{definition reflexive differentials}
 Let $X$ be a complex space. The module of reflexive $p$-forms on $X$ is defined as
 \[
  \Omega_X^{[p]}:=\left(\Omega_X^p\right)^{\vee\vee}
 \]
where $F^\vee=\Hom_{\sO_X}(F,\sO_X)$ is the dual of a sheaf of $\sO_X$-modules.
\end{definition}
If $X$ is a reduced normal complex space and $j:U \into X$ denotes the inclusion of the regular locus, then $\Omega_X^{[p]}=j_*\Omega^p_U$. For a resolution of singularities $\pi:Y \to X$ we moreover have $\pi_*\Omega^p_Y=\Omega_X^{[p]}$
by \cite[Corollary~1.7]{KS18} if in addition $X$ has rational singularities.  If finally $X$ is also of Fujiki class $\scrC$, then for $p+q\leq 2$ the graded pieces of the Hodge filtration can be identified with $H^q(X,\Omega_X^{[p]})$, see e.g. \cite[Corollary~2.3]{BL16}.
%------------------------------------------------------------------------------------------

%------------------------------------------------------------------------------------------
\subsection{Kähler spaces}
%------------------------------------------------------------------------------------------

The notion of a Kähler complex space, which we now recall, is due to Grauert \cite[\S 3, 3., p. 346]{Gra62}. Recall that a \emph{smooth function} on a complex space $Z$ is by definition just a function $f:Z\to \R$ such that under a local holomorphic embedding of $Z$ into an open set $U\subset \C^n$, there is a smooth (i.e., $C^\infty$) function on $U$ (in the usual sense) that restricts to $f$ on $Z$.

\begin{definition}\label{definition kaehler form}
Let $Z$ be a complex space. A \emph{Kähler form} for $Z$ is  given by an open covering $Z=\bigcup_{i\in I} U_i$ and smooth strictly plurisubharmonic functions $\vphi_i:U_i\to \R$ such that on $U_{ij}:=U_i\cap U_j$ the function $\vphi_i\vert_{U_{ij}} - \vphi_j\vert_{U_{ij}}$ is pluriharmonic, i.e., locally the real part of a holomorphic function. 
\end{definition}

There are two important sheaves related to Kähler forms. We denote by $\PH_Z$ the sheaf of pluriharmonic functions on $Z$ and by $C^\infty_Z$ the sheaf of smooth real-valued functions on $Z$. Then we have the sequences
\begin{equation}\label{eq psh sequence}
0 \to \PH_Z \to C^\infty_Z\to C^\infty_Z/\PH_Z\to 0
\end{equation}

and 

\begin{equation}\label{eq psh class sequence}
0 \to \R_Z \to[i] \sO_Z \to[R] \PH_Z \to 0,
\end{equation}
where $i$ stands for multiplication by $\sqrt{-1}$ and $R$ is given by taking the real part. Thus, a Kähler form on $Z$ gives rise to an element $\omega\in H^0(Z,C_Z^\infty/\PH_Z)$. Successively applying the connecting homomorphisms of \eqref{eq psh sequence} and \eqref{eq psh class sequence} we obtain classes $[\omega] \in H^1(Z,\PH_Z)$ and $[\omega]\in H^2(Z,\R)$. The latter is called \emph{the Kähler class of $\omega$}.

\begin{definition}\label{definition kaehler class}
Let $Z$ be a reduced complex space. A Kähler class on $Z$ is a class $\kappa\in H^2(Z,\R)$ which is the Kähler class of some Kähler form on $Z$. The Kähler cone is the set
\[
\sK_Z:=\{\alpha \in H^2(Z,\R)\mid \alpha = [\omega] \textrm{ for some Kähler form } \omega \}
\]
\end{definition}

\begin{remark} \label{remark kaehler} There are several things we wish to observe.
\begin{enumerate}
\item It follows from the definition that for a compact complex space $Z$ the Kähler cone $\sK_Z$ is open in the image of $H^1(Z,\PH_Z)\to H^2(Z,\R)$. Indeed, being strictly plurisubharmonic is stable under small perturbations and the connecting homomorphism $H^0(Z,C_Z^\infty/\PH_Z)\to H^1(Z,\PH_Z)$ is surjective as $C_Z^\infty$ is a fine sheaf.
\item We can describe the Kähler forms alternatively as follows: these are Kähler forms  $\omega$ on $Z^\reg$ in the usual sense such that for every $p\in Z$ there is an open neighborhood $p \in U \subset Z$ and a closed embedding $U \into V$ into a smooth Kähler manifold where the restriction of the Kähler form of $V$ to $U^\reg$ equals $\omega\vert_{U^\reg}$.
\item Let us observe that by applying the (real) operator $i\ddbar$ a Kähler form also gives rise to a global section of $\sA^{1,1}_{Z}$  where $\sA^{p,q}_{Z}$ denotes the sheaf of smooth $(p,q)$-forms with $\C$-coefficients on $Z$---which is defined in the same manner as the sheaf of $C^\infty$-functions. This is because $\ddbar\vphi_i=\ddbar \vphi_j$ on $U_{ij}$ as $\ddbar$ annihilates pluriharmonic functions.
The cohomology class of $\omega$ in $H^2(Z,\sA^\bullet_{Z})$ is the image of the Kähler class under the natural map induced by the morphism $\R_Z\to \sA^\bullet_{Z}$.
\end{enumerate}
\end{remark}

Let us recall the following properties of Kähler spaces. We will use throughout the text, sometimes without explicit mention.

\begin{proposition}\label{proposition kaehler space}\hspace{1in}

\begin{enumerate}
	\item Every subspace of a Kähler space is Kähler.
	\item A smooth complex space is Kähler if and only if it is a K\"ahler manifold in the usual sense.
	\item Every reduced Kähler space has a resolution of singularities by a Kähler manifold.
\end{enumerate}
\end{proposition}
\begin{proof}
This is a consequence of \cite[II, 1.3.1 Proposition]{Var89}.
\end{proof}

The proposition in particular implies that compact Kähler spaces are of Fujiki class $\scrC$ so that their singular cohomology groups carry a mixed Hodge structure. For $X\in\scrC$, we may thus define:
\begin{equation}\label{eq real hodge classes}
H^{k,k}(X,\R):=\Hom(\R(-k),H^{2k}(X,\R))=F^k H^{2k}(X,\C)\cap H^{2k}(X,\R).
\end{equation}
Note that the weights that show up in the mixed Hodge structure on $H^{k}(X,\Z)$ are $\leq k$---the argument for class $\scrC$ varieties is the same as in the algebraic case, cf. \cite[Theorem 5.39]{PS08}.

\begin{proposition}\label{proposition kaehler is 11}
Let $X$ be a reduced compact Kähler space. Then $\sK_X \subset H^{1,1}(X,\R)$.
\end{proposition}
\begin{proof}
The claim is easily verified using a construction of Ancona and Gaveau \cite{AG06} some properties of which we briefly recall. In this proof, all references are to \cite{AG06} if not mentioned otherwise. For a reduced complex space $X$, in Chapter~II.2 they construct a complex $\Lambda_X^\bullet$ which is a fine resolution of the constant sheaf $\C_X$. In fact, $\Lambda_X^\bullet$ is not unique but we may fix one such complex once and for all. A section of $\Lambda_X^\bullet$  by II.2, Definition~2.1 is a collection of differential forms (of shifted degrees) on an associated hypercovering $\{X_\ell \to X\}_{\ell \in L}$ where the $X_\ell$ are smooth. In Chapter~II.3 they use this complex to construct Deligne's mixed Hodge structure on $H^k(X,\Z)$ if $X$ is Kähler (or more generally of Fujiki class $\scrC$). As discussed in Section~II.2.8, the complex $\sA_X^\bullet$ of smooth differential forms on $X$ (introduced in Remark~\ref{remark kaehler} above) is a subcomplex of $\Lambda_X^\bullet$ and this inclusion clearly sends the filtration $F^p\sA_X^k:=\oplus_{r\geq p}\sA_X^{r,k-r}$ to the Hodge filtration. For a Kähler form $\omega=\{\vphi_i\}_{i\in I}\in H^0(X,C^\infty_X/\PH_X)$, the claim now follows, because $\sqrt{-1}\di\dibar\{\vphi_i\}_{i\in I} \in F^1\sA_X^1(X)$.
\end{proof}
Observe that if in addition $X$ has rational singularities, the claim of the proposition simply follows from Lemma~\ref{lemma peternell} and strictness of the pullback for the Hodge filtration.

%------------------------------------------------------------------------------------------
\subsection{Big and nef cohomology classes}\label{section ddbar cohomology}
%------------------------------------------------------------------------------------------

We briefly recall the definition of $\ddbar$-cohomo\-logy for a complex manifold $X$. As before, we denote $\sA_X^k$ respectively $\sA_X^{p,q}$ the sheaf of differential $k$-forms respectively $(p,q)$-forms with values in $\C$. Then $\ddbar$-cohomology is defined as
\begin{equation}\label{eq ddbar cohomology}
H^{p,q}_\ddbar(X):=\frac{\ker\left(d:\sA_X^{p,q}(X)\to \sA_X^{p+q+1}(X)\right)}{\img\left(i\ddbar:\sA_X^{p-1,q-1}(X)\to \sA_X^{p,q}(X)\right)}.
\end{equation}

Similarly, we write $H^{p,p}_\ddbar(X,\R)$ if we take cohomology of $\R$-valued differential forms (which is different from zero only for $p=q$). Note that $i\ddbar$ in the above formula defines a real operator.

In algebraic geometry, bigness and nefness are important notions for line bundles. In the complex analytic world, these notions can also be defined for real cohomology classes as we now recall.

\begin{definition}\label{definition big and nef}
Let $X$ be a compact complex manifold. A cohomology class $\alpha \in H_\ddbar^{1,1}(X,\R)$ is called \emph{nef} if for some hermitian form $\omega$ on $X$ and for every $\veps>0$ it can be represented by a smooth $(1,1)$-form $\eta_\veps$ such that $\eta_\veps\geq -\veps\omega$. 
A \emph{Kähler current} is a closed positive $(1,1)$-current $T$ such that $T \geq \omega$ in the sense of currents. A class $\alpha \in H_\ddbar^{1,1}(X)$ is called  \emph{big} if it can be represented by a Kähler current. 
\end{definition}

We refer to \cite[Chapter 3, 1.]{GH94} or \cite[Chapter 1]{Dem12} for a general reference on currents and notions of positivity.

\begin{remark}\label{remark ddbar for class c manifolds}
On compact manifolds of Fujiki class $\scrC$ (in particular on compact Kähler manifolds) the natural map from $\ddbar$-cohomology to de Rham cohomology is injective and gives an identification of $H^{p,q}_\ddbar(X)$ with $H^{p,q}(X)$. This follows directly from the $\di\dibar$-Lemma, see e.g. \cite[(5.21) and (5.22) Theorem]{DGMS75} for manifolds of class $\scrC$.
\end{remark}
%------------------------------------------------------------------------------------------

%------------------------------------------------------------------------------------------
\subsection{$\Q$-factoriality}\label{paragraph qfact}
%------------------------------------------------------------------------------------------

Let us spend a moment to discuss the notion of $\Q$-factoriality. A normal algebraic variety $Z$ is called \emph{$\Q$-factorial} if for every Weil divisor $D$ on $Z$ there is $m\in \N$ such that $mD$ is Cartier.  In the algebraic category, $\Q$-factoriality is local for the Zariski topology.  Recall from \cite[Proposition~2.7]{Har94} that Weil divisor classes are in bijective correspondence with isomorphism classes of reflexive sheaves of rank one: to a Weil divisor $D$ on $Z$ one associates the sheaf $\sO_Z(D)$ defined by
 \[
  U\mapsto \sO_Z(D)(U):=\{ f \in \C(Z)\mid D\vert_U + \div\left(f\vert_U\right) \geq 0\},
 \]
which is easily seen to be reflexive. So $\Q$-factoriality can be equivalently characterized using reflexive sheaves.

Finally, assume that $Z$ is compact, let $\pi:Z' \to Z$ be a resolution of singularities, and let $E_1, \ldots, E_m$ be the prime divisors contained in the exceptional locus $\Exc(\pi)$.  By \cite[(12.1.6) Proposition]{KM92}, the variety $Z$ is $\Q$-factorial if and only if 
\begin{equation}\label{eq qfactoriality characterization}
\img\left( H^2(Z',\Q) \to H^0\left(Z,R^2\pi_*\Q_{Z'}\right)\right) = \img\left( \bigoplus_{i=1}^m \Q[E_i] \to H^0(Z,R^2\pi_*\Q_{Z'})\right).
\end{equation}
See also \cite[\S 2 (i)]{Nam06} for an argument for the \emph{only if}-direction. We summarize:

\begin{lemma}\label{lemma algebraic qfactoriality}
Let $Z$ be a normal algebraic variety over $\C$. Then the following are equivalent.
\begin{enumerate}
	\item\label{lemma algebraic qfactoriality item qfactorial} $Z$ is $\Q$-factorial.
	\item Every Zariski open subset $U \subset Z$ is $\Q$-factorial.
	\item\label{lemma algebraic qfactoriality item reflexive} For every reflexive sheaf $L$ on $Z$ of rank $1$, there is $n\in \N$ such that $\left(L^{\tensor n}\right)^{\vee\vee}$ is invertible.
\end{enumerate}
If in addition $Z$ is compact and has rational singularities, the above statements are equivalent to
\begin{enumerate}
\setcounter{enumi}{3}
	\item The equality \eqref{eq qfactoriality characterization} holds for some resolution $\pi:Z'\to Z$.
\end{enumerate}
\end{lemma}
\begin{proof}
For the equivalence of \eqref{lemma algebraic qfactoriality item qfactorial} and \eqref{lemma algebraic qfactoriality item reflexive} one only needs that for a Weil divisor $D$ on $Z$ we have 
\[
\sO_Z(nD)= (\underbrace{\sO_Z(D) \tensor \ldots \tensor \sO_Z(D)}_{n-\textrm{times}})^{\vee\vee}
\]
which can be obtained by pushforward and the fact that it holds on the regular part. 
\end{proof}

In the analytic category, the situation is a little more subtle. We have several different notions which turn out to be non-equivalent, see Proposition \ref{proposition analytic qfactoriality} and Example~\ref{example qfactoriality}.

\begin{definition}\label{definition qfactorial}
A normal complex analytic variety $Z$ is called \emph{divisorially $\Q$-factorial} if for every Weil divisor $D$ on $Z$ there is $m\in \N$ such that $mD$ is Cartier and it is called \emph{$\Q$-factorial} if for every reflexive sheaf $L$ on $Z$ of rank $1$, there is $n\in \N$ such that $\left(L^{\tensor n}\right)^{\vee\vee}$ is invertible. We say that $Z$ is \emph{locally analytically (divisorially) $\Q$-factorial} if every open set $U\subset X$ in the Euclidean topology is (divisorially) $\Q$-factorial.
\end{definition}

Clearly, local analytic (divisorial) $\Q$-factoriality implies (divisorial) $\Q$-factoriality. The converse however is not true. The reason is that there are usually many more local divisors than global divisors, e.g. one cannot obtain a global divisor by taking the closure of a divisor on a small open subset. There might be no global divisors at all, see e.g. Example \ref{example qfactoriality}, which is also the reason why divisorial $\Q$-factoriality is not the right property to ask for and one should rather work with $\Q$-factoriality (defined in terms of rank one reflexive sheaves). 

\begin{proposition}\label{proposition analytic qfactoriality}
Let $Z$ be a normal complex analytic variety and consider the following statements:
\begin{enumerate}
%	\item\label{proposition analytic qfactoriality item kollarmori} The equality \eqref{eq qfactoriality characterization} holds for some resolution $\pi:Z'\to Z$.
	\item\label{proposition analytic qfactoriality item local linebundles} $Z$ is locally analytically $\Q$-factorial.
	\item\label{proposition analytic qfactoriality item local divisors} $Z$ is locally analytically divisorially $\Q$-factorial.
	\item\label{proposition analytic qfactoriality item linebundles} $Z$ is $\Q$-factorial.
	\item\label{proposition analytic qfactoriality item divisors} $Z$ is divisorially $\Q$-factorial.
\end{enumerate}
Then we have the following implications:
\[\xymatrix{
\eqref{proposition analytic qfactoriality item local linebundles} \ar@{=>}[r]\ar@{=>}[d] & \eqref{proposition analytic qfactoriality item linebundles} \ar@{=>}[d]\\
\eqref{proposition analytic qfactoriality item local divisors} \ar@{=>}[r] & \eqref{proposition analytic qfactoriality item divisors}\\
}\]

Moreover, suppose $Z$ is also compact of class $\scrC$ with rational singularities.  Then $Z$ is $\Q$-factorial if and only if for some resolution $\pi:Z'\to Z$ we have
\begin{equation}\label{eq weaker condition}
  \img\left( \Pic(Z')_\Q \to H^0(Z,R^2\pi_*\Q_{Z'})\right)=\img\left( \bigoplus_{i=1}^m \Q[E_i] \to H^0(Z,R^2\pi_*\Q_{Z'})\right).
\end{equation}
\end{proposition}
\begin{proof}
The implications \eqref{proposition analytic qfactoriality item local linebundles} $\implies$ \eqref{proposition analytic qfactoriality item linebundles} $\implies$ \eqref{proposition analytic qfactoriality item divisors} and \eqref{proposition analytic qfactoriality item local linebundles} $\implies$ \eqref{proposition analytic qfactoriality item local divisors} $\implies$ \eqref{proposition analytic qfactoriality item divisors} are immediate. 

The last part is a slight adaption of Koll\'ar-Mori \cite[(12.1.6) Proposition]{KM92}, replacing \eqref{eq qfactoriality characterization} with \eqref{eq weaker condition} which is what is actually used there.  Briefly, if $Z$ is $\Q$-factorial, then for any line bundle $M$ on $Z'$, the sheaf $L:=(\pi_*M)^{\vee\vee}$ is reflexive and therefore $\pi^*((L^k)^{\vee\vee})\cong M^k(E)$ for some divisor $E$ whose support is contained in the exceptional locus.  Hence, $\Pic(Z')_\Q=\pi^*\Pic(Z)_\Q+\sum_i\Q[E_i]$, which implies \eqref{eq weaker condition}.  Conversely, if \eqref{eq weaker condition} is satisfied, then for any rank one reflexive sheaf $L$ on $Z$ we can find a divisor $E$ whose support is contained in the exceptional locus and for which $M:=(\pi^*L)^{\vee\vee}(E)$ is numerically trivial on fibers.  But then by \cite[(12.1.4) Proposition]{KM92}, $\pi_*(M^k)$ is a line bundle for some $k$, and therefore by normality $(L^k)^{\vee\vee}$ is invertible.
\end{proof}
%------------------------------------------------------------------------------------------

%------------------------------------------------------------------------------------------
%------------------------------------------------------------------------------------------
%------------------------------------------------------------------------------------------
\section{Symplectic varieties}\label{section symplectic}
%------------------------------------------------------------------------------------------
%------------------------------------------------------------------------------------------
%------------------------------------------------------------------------------------------

For the remainder of this paper, we will use the term (primitive) symplectic variety in the following sense.

\begin{definition}
Following Beauville \cite{Bea00}, a \emph{symplectic variety} is a pair $(X,\sigma)$ consisting of a normal variety $X$ and a closed holomorphic symplectic form $\sigma \in H^0(X^\reg,\Omega_X^2)$ on $X^\reg$ such that there is a resolution of singularities $\pi:Y \to X$ for which $\pi^*\sigma$ extends to a holomorphic form on $Y$. A \emph{primitive symplectic variety} is a normal compact K\"ahler variety $X$ such that $H^1(X,\sO_X)=0$ and $H^0(X^\reg,\Omega_{X}^2)= \C \sigma$ such that $(X,\sigma)$ is a symplectic variety.
\end{definition}

Greb--Kebekus--Peternell introduced a notion of irreducible holomorphic-symplectic variety (more restrictive than ours) in \cite[Definition 8.16]{GKP16} which serves as one of the three building blocks in a decomposition theorem (due to Druel--Greb--Guenancia--Höring--Kebekus--Peternell, see introduction for references). Matsushita \cite[Definition 1.6]{Mat15} introduced the related notion of cohomologically irreducible symplectic varieties. The definition we use here appeared before in Schwald \cite[Definition 1]{Sch17} for projective varieties under the name irreducible symplectic. We chose to work with the above definition because it seems to be the most general framework that allows for a general moduli and deformation theory similar to the smooth case. We prefer however the name \emph{primitive} over irreducible symplectic for the lack of a decomposition theorem. This fits together with Menet's usage \cite[Definition 3.1]{Men20}.

\begin{Ex}\label{exhs}\

\begin{enumerate}
\item If $X$ is a primitive symplectic variety then so is:
\begin{itemize}
\item any contraction, that is, $X'$ for any proper bimeromorphic $f:X\to X'$ onto a normal K\"ahler space;
\item any quotient of $X$ by a finite group of symplectic automorphisms \cite[Proposition 2.4]{Bea00};
\item any small locally trivial deformation (see Corollary~\ref{cor deformation symplectic} below).
\end{itemize}
\item \label{egnikulin} By Nikulin \cite{NikulinAuto} any symplectic involution $\iota$ of a K3 surface $S$ has 8 fixed points.  The quotient $X$ of the Hilbert scheme $S^{[n]}$ of $n\geq 3$ points by $\iota$ has $\Q$-factorial terminal singularities by \cite[Proposition~5.15]{KM98} and Theorem~\ref{theorem basic symplectic}\eqref{lemma basic symplectic part two} below.  

For $n=2$, $X$ has $\binom{8}{2}=28$ isolated singularties and a K3 surface of transverse $A_1$ singularties, corresponding to the 28 fixed reduced subschemes and the closure of the locus of reduced orbits, respectively (see for example \cite[\S 6]{Camere}). It is therefore not terminal.  The $\Q$-factorial terminalization $Y$ is obtained by blowing up the K3 surface.  The second Betti number of $X$ is 15, and so the locally trivial deformation space of $X$ is 13-dimensional while $Y$ deforms in one dimension higher (see Theorem \ref{theorem deflt is smooth} below).  A complete projective family of this deformation type is produced in \cite{newnikulin}; see \cite{menetfu} for some other ``new" deformation types.
\item  There is a cubic fourfold $Z\subset\P^5$ with an order 11 automorphism (see for example \cite{MonAuto}).  Its Fano variety of lines $F$ has a symplectic automorphism $\sigma$ with isolated fixed points, and the quotient $X=F/\sigma$ is a $\Q$-factorial terminal primitive symplectic variety with $b_2=3$.  It follows from \cite[Theorem 3.17 and Theorem 5.4]{Men20} that the only deformation of $X$ is the twistor deformation.
\item \label{egsingmod} Let $S$ be a projective K3 surface, and $v\in H^*(S,\Z)$ an algebraic Mukai vector with $v^2>0$.  Then for $k\geq 1$, the moduli space $X=M(kv)$ of stable sheaves of Mukai vector $kv$ with respect to a generic polarization is a primitive symplectic variety.  Moreover, $X$ is always locally factorial and terminal \cite[Theorem~A]{KLSsing} unless $k=2$ and $v^2=2$ (in which case $X$ admits a resolution by an irreducible symplectic manifold---the O'Grady 10-fold \cite{OG99}). The singularities of $M(kv)$ can be non-quotient singularities, as the completions of the local rings are often not (even analytically) $\Q$-factorial---see \cite[Remark~6.3]{KLSsing}. This is because analytically locally or \'etale locally, these examples admit small crepant resolutions (but not globally).
\item \label{egsacca} Forthcoming work of Sacc\`a \cite{saccacpt} shows using MMP techniques that a (projective) Lagrangian fibration which extends in codimension 2 admits a compactification as a $\Q$-factorial terminal symplectic variety.  This for example applies to show that if $f:X\to B$ is a Lagrangian fibration of a smooth (projective) irreducible symplectic variety which is smooth over $U\subset B$, then any fibration isogenous to $f^{-1}(U)\to U$ admits such a compactification.
\item \label{egcubic} For a possibly singular cubic fourfold $Y \subset  \P^5$ not containing a plane, it was shown in \cite[Theorem 3.3]{Leh18} that the variety $M_1(Y)$ of lines on $Y$ is a symplectic variety birational to the second punctual Hilbert scheme of an associated K3 surface. Hence, $M_1(Y)$ admits a crepant resolution by an irreducible symplectic manifold, see \cite[Corollary 5.6]{Leh18}. A similar statement is deduced for the target space $Z(Y)$ of the MRC-fibration of the Hilbert scheme compactification of the space of twisted cubics on $Y$, see Theorem~1.1, Corollary~5.5,  and Corollary~6.2 of \cite{Leh18}.
\end{enumerate}
\end{Ex}

Note that even for smooth $X$ the notion of a primitive symplectic variety is a priori more general than that of an irreducible symplectic manifold. However, we do not know if there are smooth primitive symplectic varieties which are not irreducible symplectic manifolds. By Lemma \ref{lemma isv is ism} below such a variety must have dimension $\geq 6$.

\begin{lemma}\label{lemma isv is ism}
Let $X$ be a smooth primitive symplectic variety of dimension $\leq 4$. Then $X$ is an irreducible symplectic manifold (in the classical sense).
\end{lemma}
\begin{proof}
For $\dim X=2$ this is well known, so let us assume $\dim X=4$. 

If $X$ is a smooth primitive symplectic variety in our sense, the Beauville-Bogomolov decomposition theorem yields that a finite topological cover $\wt X \to X$ of $X$ splits as a product $\wt X \isom H \times C \times T$ where $H$ is a product of irreducible symplectic manifolds, $C$ a product of strict Calabi-Yau varieties, and $T$ a complex torus. From the existence of a symplectic form on $\wt X$ (by pullback from $X$) we deduce that the factor $C$ is trivial. 

By assumption, $H^1(\sO_X)=0$ and thus $H^3(\sO_X)=0$ by Serre duality.  Moreover, by the unicity of the symplectic form we in fact have $\chi(\sO_X)=3$. If there is a torus factor, then $\chi(\sO_{\wt X})=0$ contradicting $\chi(\sO_{\wt X})=d \chi(\sO_{X})$ where $d$ is the degree of the cover, so the factor $T$ is trivial. If $\wt X$ is a product of K3 surfaces, then $\chi(\sO_{\wt X})=4$, which is impossible.  Thus, $\wt X$ is irreducible symplectic, so that $d=1$, and thus $X$ is irreducible symplectic as well.
\end{proof}

It is unclear whether the statement of Lemma \ref{lemma isv is ism} holds in higher dimensions. It is worthwhile noting that there is a singular example of a primitive symplectic variety due to Matsushita \cite{Mat01}, see also \cite[Lemma 15]{Saw14} and \cite[Example 29]{Sch17}, which has the right cohomological invariants but is a torus quotient. Schwald's account nicely illustrates how the geometry of primitive symplectic varieties may deviate from the one of irreducible symplectic manifolds.

We collect the following basic results about symplectic varieties which are due to work of Beauville, Kaledin, and Namikawa; we give precise references in the proof.

\begin{theorem}[Beauville, Kaledin, and Namikawa]\label{theorem basic symplectic} \hspace{1in}

\begin{enumerate}
\item A normal variety is symplectic if and only if it has only rational Gorenstein singularities and its smooth part admits a holomorphic symplectic form. In particular, a symplectic variety has rational singularities.
\item Let $X$ be a symplectic variety and consider the stratification $$X=X_0 \supset X_1 \supset \ldots$$ where $X_{i+1}$ is the singular part of $X_i$ endowed with the reduced structure. Then the normalization of every irreducible component of $X_i$ is a symplectic variety. In particular, the singular locus of a symplectic variety has even codimension.
	\item \label{lemma basic symplectic part two} A symplectic variety $X$ has terminal singularities if and only if $\codim_X X^\sing \geq~4$.
\end{enumerate}
\end{theorem}
\begin{proof} 
At least for algebraic varieties, this result is well-known. We give a sketch of the argument and comment on why the arguments hold in the analytic context as well.
\begin{enumerate}
	\item The \emph{only if} direction is proven in \cite[Proposition 1.3]{Bea00} and is valid in the analytic context as well. The converse follows from \cite[Corollary~1.7]{KS18}.
	\item The existence of the stratification is \cite[Theorem 2.3]{Kal06b}. It is not claimed there that $X_{i+1}=(X_{i}^\sing)_\red$, however, that is how the stratification is constructed, see \cite[Proposition 3.1]{Kal06b}. The decomposition a priori only holds on the formal level by Kaledin's result, however by \cite[Corollary (1.6)]{Art68} a formal isomorphism implies the existence of an isomorphism of analytic germs. 	
	We refer to Remark \ref{remark kaledin} for why Kaledin's results also apply in the analytic situation. 
\item For algebraic varieties, this statement is \cite[Corollary 1]{Nam01}. The proof is a bit involved so we take the opportunity to use Kebekus-Schnell's functorial pullback of reflexive differential forms and Kaledin's decomposition theorem to write down a simple proof that also works in the analytic setting. We do not claim originality, the argument expands on an observation by Namikawa (see the footnote on page 1 and \S 1 of \cite{Nam01}).

By \cite[Theorem 2.3]{Kal06b}, the codimension of the singular locus is even, and if $x \in X^\sing$ is a general point of an irreducible component of $X^\sing$ of codimension $2$, the germ $(X,x)$ is isomorphic to the product of a smooth germ and the germ of rational double point. Such a product however does not have terminal singularities. If $\codim X^\sing \geq 4$, we take a resolution $\pi:Y\to X$ and assume that $E\subset Y$ is a divisor with vanishing discrepancy. Then $Y$ is symplectic at the generic point of $E$ and $\pi(E) \subset X^\sing$. Let us consider a diagram
\[
\xymatrix{
E' \ar[d]_{\pi'}\ar[r]\ar@/_1pc/[rr]_\psi & E \ar@{^(->}[r] & Y \ar[d]^\pi \\
\Sigma \ar[rr]_\vphi && X\\
}
\]
where $\Sigma$ is a resolution of $\pi(E)$ and $E' \to E$ is a resolution. Then by \cite[Theorem~14.1]{KS18} one can pullback the symplectic form along $\vphi$ such that ${\pi'}^*\vphi^*\sigma = \psi^*\pi^*\sigma$. The pullback $\psi^*\pi^*\sigma$ has one-dimensional radical at the general point of $E'$ and $\vphi^*\sigma$ is generically symplectic by Kaledin's result. This is a contradiction to $\dim \Sigma \leq \dim X - 4$.
\end{enumerate}
\end{proof}

As a direct consequence of Theorem~\ref{theorem basic symplectic} and Lemma \ref{lemma peternell} we infer

\begin{corollary}\label{corollary symplectic hodge}
Let $X$ be a compact symplectic variety. Then the Hodge structure on $H^2(X,\Z)_\tf$ is pure.\qed
\end{corollary}
%------------------------------------------------------------------------------------------

%------------------------------------------------------------------------------------------

\begin{remark}\label{remark kaledin}
Kaledin's article \cite{Kal06b} is formulated for complex algebraic varieties, but his results are used in Theorem~\ref{theorem basic symplectic} for arbitrary  symplectic varieties.
Let us comment on why they carry over to the analytic setting. The crucial ingredient from \emph{algebraic} geometry in Kaledin's proofs is the use of functorial mixed Hodge structures on cohomology groups of complex projective algebraic varieties and there is no such structure on the cohomology of arbitrary complex varieties. However, Kaledin only uses it for fibers of resolutions of singularities which, also in the analytic category, can be chosen projective. Actually, these fibers are always compact complex varieties of Fujiki class $\scrC$, which is sufficient.

With this in mind, Kaledin's proofs work almost literally for analytic varieties. More precisely, one first shows using mixed Hodge structures that Kaledin's proofs yield analogs of \cite[Lemma 2.7]{Kal06b} and \cite[Lemma 2.9]{Kal06b} in the analytic setting. These are the key technical ingredients to prove the stratification and formal product decomposition \cite[Theorem 2.3]{Kal06b} as well as \cite[Theorem 2.5]{Kal06b} which relates the symplectic and Poisson structure. 
Other than mixed Hodge theory, Kaledin mainly uses Poisson structures, commutative algebra, or direct geometric arguments which all make sense also in our setting.
Finally, also semi-smallness (see \cite[Lemma~2.11]{Kal06b}) is a consequence of geometric properties of the symplectic form and \cite[Lemma~2.9]{Kal06b}.
\end{remark}
%------------------------------------------------------------------------------------------

%------------------------------------------------------------------------------------------
\section{Deformation theory}\label{section defo}
%------------------------------------------------------------------------------------------

\begin{definition}\label{definition deformation}
A deformation of a compact complex space $Z$ is a flat and proper morphism $\scrZ \to S$ of complex spaces together with a distinguished point $0\in S$ and an isomorphism of the fiber of $\scrZ\to S$ over $0$ with $Z$. A deformation $\pi:\scrZ \to S$ is called \emph{locally trivial} at $0\in S$ if for every $p \in Z=\pi^{-1}(0)$ there exist open neighborhoods $\scrU \subset \scrZ$ of $p$ and $S_0 \subset S$ of $0$ such that $\scrU \isom U \times S_0$ over $S_0$ where $U=\scrU \cap Z$. The deformation is called locally trivial if it is locally trivial at each point of $S$. We speak of a \emph{locally trivial family} or \emph{locally trivial morphism} $\pi:\scrZ \to S$ if we do not specify $0\in S$ and the fiber over it.
\end{definition}
For most properties and statements we should rather speak about the morphism of space germs $(\scrZ,Z) \to (S,0)$. All deformation theoretic statements have to be interpreted as statements about germs. Considering deformations and locally trivial deformations gives rise to two deformation functors; in fact, the functor $\sD_Z^\lt$ of locally trivial deformations of $Z$ is a subfunctor in the functor $\sD_Z$ of all deformations of $Z$. They have tangent spaces $T_{\sD_Z^\lt}= H^1(Z,T_Z)$ and if $Z$ is reduced $T_{\sD_Z}=\Ext^1(\Omega_Z,\sO_Z)$, respectively. Note that $H^1(Z,T_Z)$ is a subset of $\Ext^1(\Omega_Z,\sO_Z)$ by the local-to-global spectral sequence for Ext. 
We refer to \cite[Proposition~1.2.9]{Ser06} (which actually works for arbitrary schemes) respectively \cite[Theorem~2.4.1(iv)]{Ser06}. Even though Sernesi's book treats deformations of algebraic schemes, the arguments apply literally for deformations of complex spaces, mainly because zero-dimensional complex spaces are nothing else but zero-dimensional $\C$-schemes of finite type.

\subsection{Versality and Universality}\label{section versality} Recall that a deformation $(\scrZ,Z) \to (S,0)$ is called \emph{versal} if for every deformation $(\scrZ',Z)\to (S',0)$ of $Z$ there is a map $\vphi:(S',0) \to (S,0)$ of (germs of) complex spaces such that $\scrZ\times_S S' \isom \scrZ'$. It is called \emph{miniversal} if moreover the differential $T_{\vphi,0}:T_{S',0} \to T_{S,0}$ is uniquely determined. The deformation is called \emph{universal} if furthermore the map $\vphi$ is unique. Clearly, every universal deformation is miniversal and every miniversal deformation is versal. The different notions of versality are defined analogously for other deformation problems such as locally trivial deformations. 

\subsection{Existence of a miniversal deformation}\label{section existence} Miniversal\footnote{Note that Grauert uses the term complete (resp. versal) for what we call versal (resp. miniversal). Nowadays, our terminology seems to be more common; some authors use semi-universal instead of miniversal.} deformations exist by \cite[Hauptsatz, p 140]{Gra74}, see also \cite[Th\'eor\`eme principal, p 598]{Dou74}. More precisely, it is shown in loc. cit. that there exist miniversal deformations $\scrZ \to S$ of a given compact complex space $Z$ which are versal in every point of $S$.
We will frequently write $S=\Def(Z)$. The family $\scrZ\to \Def(Z)$ is called the \emph{Kuranishi family} and $\Def(Z)$ is called \emph{Kuranishi space}. 

If $Z$ is a complex space satisfying $H^0(Z,T_Z)=0$, then every miniversal deformation is universal. 

\subsection{Locally trivial miniversal deformations}\label{section locally trivial}
Recall from \cite[(0.3) Corollary]{FK87} that for a miniversal deformation $\scrZ \to \Def(Z)$ of a compact complex space $Z$ there exists a closed complex subspace $\Def^\lt(Z)\subset \Def(Z)$ of the Kuranishi space parametrizing locally trivial deformations of $Z$. More precisely, the restriction of the miniversal family to this subspace, which by abuse of notation we denote also by $\scrZ \to \Def^\lt(Z)$, is a locally trivial deformation of $Z$ and is miniversal for locally trivial deformations of $Z$.
When speaking about locally trivial deformations we will usually use the terms versal, miniversal, universal with respect to the functor of locally trivial deformations. 
%------------------------------------------------------------------------------------------

\begin{lemma}\label{lemma tx}
Let $S$ be a complex space and let $f:\sX \to S$ be a locally trivial family whose fiber $X$ above a point $0\in S$ is a primitive symplectic variety. Denote by $j:\sU \to \sX$ the inclusion of the regular locus.  Then in a neighborhood of $0\in S$ we have:
\begin{enumerate}
\item $L:=(f\circ j)_*\Omega^2_{\sU/S}$ is an invertible sheaf and compatible with arbitrary base change.
\item The natural map
\begin{equation}
T_{\sX/S} \tensor f^*L \to j_*\Omega_{\sU/S}\label{eq tx two}
\end{equation}
is an isomorphism.
\end{enumerate}
\end{lemma}
\begin{proof}
By local triviality, the sheaves $j_*\Omega^p_{\sU/S},j_*T_{\sU/S},\Omega_{\sX/S},T_{\sX/S}$ are all flat over $S$ and compatible with arbitrary base change. As push forward is compatible with flat base change, invertibility of $L$ can be tested on the completion. By the theorem on formal functions we may reduce (1) to the case where $S$ is the spectrum of an artinian local $\C$-algebra of finite type. Then by the primitivity assumption on $X$ and \cite[Lemma~2.4]{BL16}, the sheaf $L$ is invertible and compatible with arbitrary base change in a neighborhood of $0$. As every section of $L$ determines a morphism $T_{\sU/S} \to \Omega_{\sU/S}$, we obtain a canonical morphism $j_*T_{\sU/S} \tensor f^*L \to j_*\Omega_{\sU/S}$ and \eqref{eq tx two} is just the composition with $T_{\sX/S} \to j_*T_{\sU/S}$ tensored with the pullback of $L$.  It then follows that \eqref{eq tx two} is an isomorphism in a neighborhood of $0$ because it is over the special fiber.
\end{proof}

\begin{lemma}\label{lemma existence universal deformation}
Let $X$ be a primitive symplectic variety. Then $H^0(X,T_X)=0$ and every miniversal deformation of $X$ is universal.
\end{lemma}
\begin{proof}
Let $\pi:Y \to X$ be a resolution of singularities by a Kähler manifold and denote by $j:U \into X$ the inclusion of the regular part. Then 
$T_X \isom \pi_*\Omega_Y$ by Lemma~\ref{lemma tx} and \cite[Corollary~1.8]{KS18}. Consequently, 
$$H^0(X,T_X) =H^0(Y,\Omega_Y) \isom H^{1,0}(Y)$$
by the Dolbeault isomorphism and the complex conjugate of the latter is 
$$H^{0,1}(Y)\isom H^1(Y,\sO_Y) = H^1(X,\sO_X)$$
again by Dolbeault and by rationality of singularities.  We conclude the proof with the observation that $H^1(X,\sO_X)=0$ by definition of a primitive symplectic variety.
\end{proof}

%------------------------------------------------------------------------------------------

The proof of the following result is similar to the proof of \cite[Theorem 4.1]{BL16}. For lack of a crepant resolution, some minor changes are necessary which is why we include a proof.

%------------------------------------------------------------------------------------------
\begin{theorem}\label{theorem deflt is smooth}
Let $X$ be a primitive symplectic variety. Then the space $\Def^\lt(X)$ of locally trivial deformations of $X$ is smooth of dimension $h^{1,1}(X)$. 
\end{theorem}
\begin{proof}
Smoothness is deduced using Kawamata--Ran's $T^1$-lifting principle \cite{Ran92, Kaw92, Kaw97a}, see also \cite[\S 14]{GHJ03}, \cite{Leh16}, \cite[VI.3.6]{Leh11} for more details. We have to show the following. Let $\sX \to S$ be a locally trivial deformation of $X$ where $S=\Spec R$ for some Artinian local $\C$-algebra $R$ with residue field $\C$, let $S' \subset S$ be a closed subscheme, and let $$\sX':=\sX \times_S S' \to S'$$ be the induced deformation. Then we need to prove that the canonical morphism $$H^1(T_{\sX/S})\to H^1(T_{\sX'/S'})$$ is surjective.

Let $j:\sU \into \sX$ the inclusion of the regular part. By Lemma \ref{lemma tx}, it suffices to show that $H^1(j_*\Omega_{\sU/S}) \to H^1(j'_*\Omega_{\sU'/S'})$ is surjective where $j:\sU'=\sU\times_S S'\to\sX'$ is the regular part of $\sX'\to S'$. However, by \cite[Lemma 2.4]{BL16} the $R$-module $H^1(j_*\Omega_{\sU/S})$ is locally free and compatible with arbitrary base change. In other words, $$H^1(j_*\Omega_{\sU'/S'}) = H^1(j_*\Omega_{\sU/S}) \tensor_R R',$$ where $S'=\Spec R'$ and the map is clearly surjective. Thus, it follows from the $T^1$-lifting criterion that the space $\Def^\lt(X)$ is smooth.

Recall that the tangent space to $\Def^\lt(X)$ at the origin is $H^1(T_X)\isom H^1(j_*\Omega_U)$, which by \cite[Corollary 2.3]{BL16} has dimension $h^{1,1}(X)$. By the smoothness assertion we proved before the dimension of the tangent space is the dimension of $\Def^\lt(X)$.
\end{proof}

As an application, we deduce the existence of a simultaneous resolution.
\begin{definition}\label{definition simultaneous resolution}
Let $\sX \to S$ be a flat morphism between complex spaces with reduced and connected fibers. A \emph{simultaneous resolution} of $\sX \to S$ is a proper bimeromorphic $S$-morphism $\pi:\sY \to \sX$ such that $\sY \to S$ is smooth. A simultaneous resolution is called \emph{strong} if moreover $\pi$ is an isomorphism over the complement of the singular locus of $\sX \to S$.
\end{definition}
%------------------------------------------------------------------------------------------
It follows from the definition that for every $s\in S$ the fiber $\sY_s \to \sX_s$ is a resolution of singularities. It is well known that simultaneous resolutions do not always exist. For example, let $f:\sX \to S$ be a family of elliptic curves where $\sX$ is smooth and $S$ is a smooth curve. Suppose that there is a point $0\in S$ such that $f$ is smooth over $S\ohnenull$ and $\sX_0=f^{-1}(0)$ is a reduced nodal rational curve. If there were a simultaneous resolution $\pi:\sY \to \sX$, the exceptional set of $\pi$ would be a divisor $E\subset \sY$. Then $\pi(E) \subset \sX$ would be a finite set which contradicts smoothness of $\sY \to S$ because this map would have some reducible fibers.
%------------------------------------------------------------------------------------------
\begin{lemma}\label{lemma simultaneous resolution}
Let $\sX \to S$ be a locally trivial deformation of a reduced compact complex space $X$ over a reduced complex space $S$ and let $\sU \to S$ be the regular part of $\sX \to S$. Then there exists a simultaneous resolution $\pi:\sY \to \sX$ of $\sX$ which is obtained by successive blowing ups along centers which are smooth over $S$. Moreover, $\pi$ can be chosen to be an isomorphism over $\sU$.
\end{lemma}
\begin{proof}
By \cite{BM97}, resolution of singularities works algorithmically, see also \cite{Vil89}. Given a global embedding $X\subset M$ into a smooth space $M$, Bierstone and Milman define an invariant $\iota:=\mathrm{inv}_X^e:M \to \Gamma$ with values in an ordered set in \cite[Theorem 1.14 and Remark 1.16]{BM97} such that the locus where $\iota$ is maximal is smooth and Zariski closed. As explained in \cite[proof of Theorem 1.6, p. 285]{BM97}, successively blowing up the maximal locus of $\iota$ gives an algorithmic resolution. The invariant $\iota$ a priori depends on the embedding $X \subset M$. However, it is explained in \cite[13.]{BM97} that it is in fact independent of the local embedding. It only depends on the local ring at the point and on the history of the blow up (which is how they obtain resolution results without the hypothesis of $X$ being embedded). 

Therefore, we may apply the same argument in the relative setting for locally trivial deformations. Given a point $p \in \sX$ mapping to $s\in S$, we choose neighborhoods $\sV$ of $p$ in $\sX$ and $S_0$ of $s$ in $S$ and a trivialization $\vphi:\sV \to[\isom] V \times S_0$ where $V=\sV \cap \sX_s$. The maximal locus of the Bierstone-Milman invariant $\iota$ defines a smooth closed subset $C \subset V^\sing$ of the singular locus $V^\sing \subset V$. By local triviality, the singular locus $\sV^\sing$ of $\sV \to S_0$ is identified under $\vphi$ with $V^\sing \times S_0$. Thanks to the above mentioned independence of $\iota$, the closed subsets $C \times S_0$ glue to give a center $\sC \subset \sX$ for a blow up and $\sC$ is smooth over $S$. Moreover, the blow up of $\sX$ in $\sC$ is by construction again locally trivial over $S$, hence we can repeat the process and obtain the sought-for resolution $\pi:\sY \to \sX$. 
\end{proof}
%------------------------------------------------------------------------------------------
\begin{remark}\label{remark simultaneous resolution}
As the morphism $\pi:\sY\to\sX$ from the preceding lemma is obtained by successive blow ups in centers which are smooth over $S$, every such blow up family is locally trivial over $S$ and moreover, also the morphism $\pi$ is itself locally trivial. More precisely, for every open sets $\sV \subset \sX$ and $S_0\subset S$ admitting a trivialization $\vphi:\sV \to[\isom] V \times S_0$ where $V$ is the intersection of $\sV$ with some fiber over a point of $S_0$, there is a trivialization $\phi:\pi^{-1}\left(\sV\right) \to \pi^{-1}\left(V\right) \times S_0$ such that the diagram
\[
\xymatrix{
\pi^{-1}\left(\sV\right) \ar[d]\ar[r] & \ar[d] \pi^{-1}\left(V\right) \times S_0\\
\sV \ar[r] & V \times S_0\\
}
\]
commutes (and similarly for any intermediate step of the resolution procedure). 
\end{remark}

%------------------------------------------------------------------------------------------
\begin{corollary}\label{cor deformation symplectic}Every small locally trivial deformation of a primitive symplectic variety $X$ is a primitive symplectic variety.  In particular, the locally trivial Kuranishi family of a primitive symplectic variety is universal (for locally trivial deformations) for all of its fibers.
\end{corollary}
\begin{proof}
Let $f:\sX\to S$ be a small locally trivial deformation of $X=f^{-1}(0)$, $0\in S$. First note that $X$ has canonical, hence rational singularities by Theorem~\ref{theorem basic symplectic}, so by \cite[Proposition~5]{Nam01b}, nearby fibers remain K\"ahler. 
We choose a simultaneous resolution $\pi:\sY \to \sX$ over $S$, denote by $j:\sU \to \sX$ the inclusion of the regular locus, and consider the canonical morphism $f_*\pi_*\Omega^2_{\sY/S} \to (f\circ j)_*\Omega^2_{\sU/S}$. Both sheaves are locally free and compatible with arbitrary base change, the former by the argument of \cite[Théorème~5.5]{Del68}---see e.g. \cite[Lemma~2.4]{BL16} for the necessary changes in the analytic category---the latter by Lemma~\ref{lemma tx}. As $X$ is a primitive symplectic variety, both sheaves are invertible and the above morphism is an isomorphism at the point corresponding to $X$, hence in a small neighborhood. We thus find a relative holomorphic $2$-form $\omega$ on $\sU$ whose pullback extends to a holomorphic $2$-form on $\sY$. As the restriction $\omega_0$ to the fiber $X=\scrX_0$ is nondegenerate, the same is true for the restriction $\omega_s$ to $\sX_s$ for $s\in S$ close to $0$. Hence, the nearby fibers $\sX_s$ are symplectic varieties whose symplectic form is unique up to scalars. By semi-continuity, $H^{1}(\sX_s, \sO_{\sX_s})=0$ for all $s$ in a neighborhood of $0 \in S$, and so the first claim follows. The last claim follows directly from Lemma \ref{lemma existence universal deformation} and openness of versality, see \cite[Hauptsatz, p~140]{Gra74}.
\end{proof}

%------------------------------------------------------------------------------------------

%------------------------------------------------------------------------------------------
\subsection{Deformations of line bundles}\label{section deformation of line bundles}
Let $X$ be a primitive symplectic variety and $L$ a line bundle on it. We will frequently consider deformations of the pair $(X,L)$.  For this purpose one considers the morphism $d\log : \sO^\times_X\to\Omega_X$, $f\mapsto \dfrac{df}{f} $ and the induced first Chern class morphism $$\chern_1:H^1(X,\sO_X^\times)\to H^1(X,\Omega_X)\to H^1(X,\Omega^{[1]}_X)$$ which takes values in the cohomology of reflexive differentials. Recall that we have $H^1(X,\Omega_X^{[1]})\isom H^{1,1}(X)$ by \cite[Corollary 2.3]{BL16}.

\begin{lemma}\label{lemma deformations line bundle}
Let $L$ be a nontrivial line bundle on $X$. Then the canonical projection $\Def^\lt(X,L)\to \Def^\lt(X)$ is a closed immersion and identifies $\Def^\lt(X,L)$ with a smooth hypersurface whose tangent space is equal to $$\ker\left(H^1(X,T_X)\to[\cup \ \chern_1(L)] H^2(X,\sO_X)\right)$$ where the map is given by contraction and cup product.
\end{lemma}
\begin{proof}
We have a canonical map $$H^1(X,\Omega^{[1]}_X) = \Ext_X^1(\sO_X,\Omega_X^{[1]})\to \Ext_X^1(T_X,\sO_X)$$ given by sending an extension to its dual (observe that $\sExt_X^1(\sO_X,\sO_X)=0$). Therefore, $\chern_1(L) \in H^1(X,\Omega_X^{[1]})$ gives rise to an extension 
\begin{equation}\label{eq deformations of linebundles}
0\ \to \sO_X\to E_L \to T_X\to 0
\end{equation}
and the sheaf  $E_L$ is shown to control the deformation theory of the pair $(X,L)$ in the sense that $H^1(X,E_L)$ is the tangent space to the functor $\sD_{(X,L)}$ of deformations of the pair $(X,L)$ and $H^2(X,E_L)$ is an obstruction space, see e.g. \cite[Theorem 3.3.11]{Ser06}. The proof there is written for nonsingular projective varieties only, however, the argument is the same for \emph{locally trivial} deformations of compact complex spaces. The rest of the proof is exactly as in \cite[1.14]{Huy99}.
\end{proof}
%------------------------------------------------------------------------------------------

%------------------------------------------------------------------------------------------
%------------------------------------------------------------------------------------------
%------------------------------------------------------------------------------------------
\section{The Beauville--Bogomolov--Fujiki form and local Torelli}\label{section BBF}
%------------------------------------------------------------------------------------------
%------------------------------------------------------------------------------------------
%------------------------------------------------------------------------------------------

In this section, we develop the theory of the Beauville--Bogomolov--Fujiki (BBF) form for primitive symplectic varieties. Thanks to previous works by several authors (see Section~\ref{section bbf details}) such a form exists and was known to share many properties with its counterpart in the smooth case. After a brief summary of these results with no claim for originality, the first fundamentally new result is the local Torelli theorem for locally trivial deformations, see Proposition~\ref{proposition local torelli}, which was established for $\Q$-factorial terminal varieties by Namikawa \cite[Theorem~8]{Nam01b}. With this at hand, we prove many advanced features of the BBF form that are known in the smooth case: the higher degree Fujiki relations in Proposition \ref{proposition fujiki relations}, a Riemann-Roch type formula in Corollary \ref{corollary riemann roch via q}, and the non-existence of subvarieties of odd dimension on a general deformation in Corollary \ref{corollary no odd subvarieties}.

The material developed in this section is essential in the proof of the projectivity criterion in Section \ref{section projectivity criterion}.

%------------------------------------------------------------------------------------------
\subsection{The Beauville--Bogomolov--Fujiki form}\label{section bbf details}
%------------------------------------------------------------------------------------------

Let $X$ be a primitive symplectic variety. Due to the work of Namikawa \cite{Nam01b}, Kirschner \cite{Kir15}, Matsushita \cite{Mat15}, and Schwald \cite{Sch17} there is a nondegenerate quadratic form $q_X:H^2(X,\R) \to \R$ whose associated bilinear form has signature $(3,b_2(X)-3)$. As for irreducible symplectic manifolds, we will refer to $q_X$ as the \emph{Beauville--Bogomolov--Fujiki (BBF) form}, see Definition \ref{definition bbf form}. We will use it to establish a local Torelli theorem in Proposition \ref{proposition local torelli} and we will see in Proposition \ref{proposition fujiki relations} that it satisfies analogous Fujiki relations as it does for irreducible symplectic manifolds.

We will first recall the following definition, see \cite[Definition 3.2.7]{Kir15} and also \cite[Definition 20]{Sch17}.

\begin{definition}\label{definition bbf form unnormalized}
Let $X$ be a compact complex variety of Fujiki class $\scrC$ and dimension $2n$ with rational singularities let $\sigma \in H^{2,0}(X)$ be the cohomology class of a holomorphic $2$-form on $X^\reg$ (recall from Lemma~\ref{lemma peternell} that the Hodge structure on $H^2(X,\Z)$ is pure). We denote by $\int_X : H^{4n}(X,\Z) \to \Z$ the cap product with the fundamental class. Then one defines a quadratic form $q_{X,\sigma}:H^2(X,\C) \to \C$ via
\begin{equation}\label{eq bbf}
q_{X,\sigma}(\alpha):= \frac{n}{2} \int_X \left(\sigma\bar\sigma\right)^{n-1}\alpha^2 + (1-n)\int_X\sigma^n\bar\sigma^{n-1}\alpha \int_X \sigma^{n-1}\bar\sigma^n\alpha.
\end{equation}
\end{definition}

If $X$ is a primitive symplectic variety, one can also define a form $q_{Y,\sigma}$ on a resolution of singularities $\pi:Y\to X$ by the analog of formula \eqref{eq bbf} where $\sigma$ is replaced by the extension of the symplectic form to $Y$ and $q_{X,\sigma}$ is the restriction to $H^2(X,\Q)\subset H^2(Y,\Q)$. This is Namikawa's approach, see \cite{Nam01b}, and both are equivalent by \cite[Corollary 22]{Sch17}. Note that Schwald assumes $X$ to be projective but this is in fact not used in the argument. 

The following result is already contained in the work of Namikawa \cite{Nam01b}, Matsushita \cite{Mat01}, Kirschner \cite{Kir15}, Schwald \cite{Sch17}. Let us emphasize that the projectivity hypothesis which is sometimes made is in fact not necessary. Denote by $b_i(X):=\dim_\Q H^i(X,\Q)$, $i \in \N_0$ the $i$-th Betti number.

\begin{lemma}\label{lemma signature bbf form}
Let $X$ be as in Definition~\ref{definition bbf form unnormalized}. Then the quadratic form 
$$q_{X,\sigma}:H^2(X,\R)\otimes H^2(X,\R)\to \R(-2)$$
 is a morphism of $\R$-Hodge structures. If $X$ is a primitive symplectic variety, $q_{X,\sigma}$ is nondegenerate and has signature $(3,b_2(X)-3)$. Furthermore, if $\sigma$ is chosen such that $\int_{X}\left(\sigma\bar\sigma\right)^n=1$, then $q_{X,\sigma}$ does not depend on~$\sigma$.  
\end{lemma}
\begin{proof}
It is immediate from \eqref{eq bbf} that $q_{X,\sigma}$ is defined over $\R$ so that the statements of the lemma make sense. The first statement is easily verified. The statement about the signature (and hence also nondegeneracy) is \cite[Theorem 2]{Sch17}. The statement about independence of $q_{X,\sigma}$ for normalized $\sigma$ is \cite[Lemma 24]{Sch17}.
\end{proof}

\begin{definition}\label{definition bbf form}
Let $X$ be a primitive symplectic variety of dimension $2n$ and let $\sigma \in H^{2,0}(X)$ be the cohomology class of a holomorphic symplectic $2$-form on $X^\reg$ satisfying $\int_{X}\left(\sigma\bar\sigma\right)^n=1$. Then the \emph{Beauville--Bogomolov--Fujiki (BBF) form} is the quadratic form $q_X:=q_{X,\sigma}$, up to scaling.
\end{definition}

%------------------------------------------------------------------------------------------
It is not hard now to deduce a local Torelli theorem for locally trivial deformations. 
Preliminary versions have been established by Namikawa \cite{Nam01a}, Kirschner \cite[Theorem 3.4.12]{Kir15}, Matsushita \cite{Mat15}, and the authors \cite{BL16}. 

\begin{proposition}[Local Torelli Theorem]\label{proposition local torelli}
Let $X$ be a primitive symplectic variety, let $q_X$ be its BBF form, and let 
\begin{equation}\label{eq local period domain}
\Omega(X):=\{[\sigma]\in \P(H^2(X,\C))\mid q_X(\sigma)=0, q_X(\sigma,\bar\sigma)>0\}  
\end{equation}
be the period domain for $X$ inside $\P(H^2(X,\C))$. If $f:\scrX \to \Def^\lt(X)$ denotes the universal locally trivial deformation of $X$ and $X_t:=f^{-1}(t)$, then the local period map
\begin{equation}\label{eq local period map}
\wp:\Def^\lt(X) \to \Omega(X), \quad t \mapsto H^{2,0}(X_t)
\end{equation}
is a local isomorphism.
\end{proposition}
\begin{proof}
Let us denote by $j:\sU \to \sX$ the inclusion of the regular locus. By Lemma~\ref{lemma tx}, the sheaf $L:=(f\circ j)_*\Omega^2_{\sU/S}$ is invertible and compatible with arbitrary base change. From this and \cite[Corollary~2.3]{BL16} we deduce that the subbundle $L\subset H^2(X,\C)\tensor \sO_{\Def^\lt(X)}$ defines the period map $\Def^\lt(X) \to \P(H^2(X,\C))$ which therefore is holomorphic. We will argue as in \cite[Théorème 5]{Bea83} to prove that it takes values in $\Omega(X)$. The statement is local, so it suffices to show that $q_X(\sigma_t)=0$ where $\sigma_t$ is a section of $f_*\Omega_{\scrX/S}^2$ evaluated at $t\in S$ for $t$ sufficiently close to the origin. This is done in the same way as in the first paragraph of the proof of \cite[Théorème 5 (b)]{Bea83}.
Let $j:U\into X$ denote the inclusion of the regular part. It is well-known that the differential of $\wp$ at zero can be described as the map
\[
H^1(X,T_X) \to \Hom(H^0(X,j_*\Omega_U^2),H^1(X,j_*\Omega_U^1))
\]
given by cup product and contraction. This is clearly an isomorphism as $H^0(X,j_*\Omega_U^2)$ is spanned by the symplectic form. Therefore, \eqref{eq local period map} is an isomorphism in a neighborhood of zero.
\end{proof}

\begin{remark}\label{remark local torelli} Namikawa assumes $\Q$-factorial terminal singularities for his local Torelli theorem \cite[Theorem~8]{Nam01b}, and in this case all deformations are locally trivial. Proposition~\ref{proposition local torelli} shows that in fact local triviality (and not the kind of singularities) is the essential ingredient.
\end{remark}

The local Torelli theorem can be exploited just as for irreducible symplectic manifolds. We start with the integrality of the quadratic form.

\begin{lemma}\label{lemma bbf topological}
The BBF form $q_{X}$ is up to a multiple a nondegenerate quadratic form $H^2(X,\Z)\to\Z$. Moreover, it is invariant under locally trivial deformations. 
\end{lemma}
\begin{proof}
The second statement is a consequence of the first, so we are left to prove integrality. This is done as in \cite[Théorème 5 (a)]{Bea83}: we deduce from the local Torelli theorem \ref{proposition local torelli} the following formula. For every $\lambda \in H^2(X,\C)$ we denote $v(\lambda):=\int_X\lambda^{2n}$ where $2n=\dim X$.   Note that for a locally trivial deformation $f:\mathscr{X}\to S$ of $X$, if $\lambda$ is a section of $R^2f_*\C$ then $v(\lambda)$ is locally constant as it can be computed on a simultaneous resolution.  For every $\alpha \in H^2(X,\C)$ we have 
\begin{equation}\label{eq integral form}
 v(\lambda)^2 q_X(\alpha) = q_X(\lambda)\left( (2n-1)v(\lambda)\int_X\lambda^{2n-2}\alpha^2 - (2n-2)\left(\int_X\lambda^{2n-1}\alpha\right)^2\right)
\end{equation}
This formula immediately shows that some real multiple of $q_X$ is defined over $\Z$.
\end{proof}

\begin{remark}As a consequence of Lemma \ref{lemma bbf topological}, we will always normalize the BBF form $q_X$ so it is a (usually primitive) integral form. 
\end{remark}

For the sake of completeness, let us summarize a statement that is well known in the smooth case.

\begin{corollary}\label{corollary deformations of line bundles}
Let $X$ be a primitive symplectic variety and let $L$ be a line bundle on it. Under the local isomorphism $\Def^\lt(X)\to \Omega(X)$ by the period map, the subspace  $\Def^\lt(X,L)$ of deformations of the pair $(X,L)$ is identified with $\P(\chern_1(L)^\perp)\cap \Omega(X)$. \qed
\end{corollary}
We will frequently simply write $\alpha^\perp$ instead of $\P(\alpha^\perp)\cap \Omega(X)$ for a class $\alpha\in H^2(X,\C)$.  
%------------------------------------------------------------------------------------------

%------------------------------------------------------------------------------------------
\subsection{A theorem of Verbitsky}
%------------------------------------------------------------------------------------------
Let $X$ be a primitive symplectic variety of dimension $2n=\dim X$.  In Section~\ref{section proj degen} we will need the following analog of a theorem of Verbitsky \cite[Theorem 1.5]{Ver96} (see also \cite{Bog96} and \cite[Proposition 24.1]{GHJ03}):
\begin{proposition} \label{sym} Let $S^*H^2(X,\C)$ be the image of the cup product map 
$$\Sym^* H^2(X,\C)\to H^{*}(X,\C).$$
  Then
\[S^*H^2(X,\C)\cong \Sym^*H^2(X,\C)/\langle x^{n+1}\mid q_X(x)=0\rangle.\]
\end{proposition}
\begin{proof}The proof in \cite{Bog96} carries through with very mild modifications, and we summarize the main points.  We have the following purely algebraic fact:
\begin{lemma}\label{isotropic fact}
Let $(H,q)$ be a complex vector space with a nondegenerate quadratic form $q$, and let $A^*$ be a graded quotient of $\Sym^*H$ by a graded ideal $I^*$ such that:
\begin{enumerate}
\item $A^{2n}\neq 0$;
\item $I^*\supset \langle x^{n+1}\mid q(x)=0\rangle$.
\end{enumerate}
Then $I^*=\langle x^{n+1}\mid q(x)=0\rangle$.	
\end{lemma}
Take $(H,q)=(H^2(X,\C),q_X)$ and $A^*=S^*H^2(X,\C)$.  Observe that the first condition in the lemma is met.  Indeed, let $w$ be a generator of the $H^{2,0}$-part of $H^2(X,\C)$.  Since for any resolution $\pi:Y\to X$ we have an injection $\pi^*:H^2(X,\C)\to H^2(Y,\C)$, it follows that $\pi^*w$ is the class of an extension of a symplectic form.  As $(\pi^*w)^n(\xxoverline{\pi^*w})^n\neq 0$, we then have that $w^n\xxoverline{w}^n\neq 0$.

Thus, it remains to verify the second condition.  We have the following:
\begin{lemma}$w^{n+1}=0$.\label{contains L}
\end{lemma}
\begin{proof}For a resolution $\pi:Y\to X$, the map $\pi^*:\mathrm{gr}^W_mH^m(X,\C)\to H^m(Y,\C)$ is injective.  Thus, the $(m,0)$-part of the mixed Hodge structure on $H^m(X,\C)$ is 0 for $m>2n$.
\end{proof}
To finish, just as in \cite{Bog96}, since the period map is an \'etale map of $\Def^\lt(X)$ onto the irreducible quadric $(q_X=0)$ by Proposition \ref{proposition local torelli}, applying Lemma \ref{contains L} to nearby deformations yields $(q_X(x)=0)\subset (x^{n+1}=0)$.
\end{proof}

%------------------------------------------------------------------------------------------
\subsection{Fujiki relations}
%------------------------------------------------------------------------------------------
Fujiki \cite[Theorem 4.7]{Fuj87} first established interesting relations between the self intersection of a given cohomology class and powers of the BBF form on symplectic manifolds. It seems that Matsushita \cite[Theorem 1.2]{Mat01}, \cite[Proposition 4.1]{Mat15} was the first to prove the ($k=\dim X$) Fujiki relation in the singular setting. He required the varieties to be projective and to have $\Q$-factorial, terminal singularities only and Schwald extended his statement to projective primitive symplectic varieties in \cite{Sch17}. We need a more general statement for the projectivity criterion in the next paragraph. Generalizing to the Kähler setup is not difficult, basically the existing proofs in the projective case work literally. 

A small argument instead is needed when comparing powers of the BBF form to integration over certain very general homology classes. The first results in this direction in the singular case can be found in \cite[Lemma 2.4]{Mat01}. 
\begin{proposition}[Fujiki relations]\label{proposition fujiki relations}
Let $X$ be a primitive symplectic variety and let $\phi \in \Sym^{k}H^2(X,\Q)^\vee$ which is of type $(-k,-k)$ for all small deformations of $X$.  Then if $k$ is odd we have $\phi=0$, while if $k$ is even there exists a constant $c=c(\phi) \in \Q$ such that $\phi=cq_X^{k/2}$, where $q_X^{k/2}\in\Sym^{k/2}H^2(X,\R)^\vee$ is the symmetrization of $q_X^{\otimes k/2}$.  In particular, for all $\alpha \in H^2(X,\C)$ we have
\begin{equation}\label{eq fujiki relations}
\phi( \alpha^{k}) =c\cdot q_X(\alpha)^{k/2}.
\end{equation}
\end{proposition}

\begin{proof}  Using Proposition \ref{proposition local torelli}, we see that the Mumford--Tate group of $H^2(X',\Z)$ for a very general locally trivial deformation $X'$ of $X$ is $\SO(H^2(X',\Q),q_{X'})$. The representation of $\SO(H^2(X,\Q),q_{X})$ on on $\Sym^kH^2(X,\Q)^\vee$ has no invariants for odd $k$, while for even $k$ the only invariant is $q_X^{k/2}$ up to scaling.
\end{proof}

\begin{corollary}\label{corollary riemann roch via q}  Let $X$ be a primitive symplectic variety.  There is a (unique) polynomial $f_X(t)\in \Q[t]$ such that for any line bundle $L$ on $X$, $\chi(L)=f_X(q_X(\mathrm{c}_1(L)))$ and $f_{X'}=f_{X}$ for any locally trivial deformation $X'$ of $X$.  Moreover, 
 
\end{corollary}
\begin{proof}  
As $X$ has rational singularities, for a resolution $\pi:Y\to X$ we have 
\[
\chi(L)=\chi(\pi^* L)=\int_Y\pi^*\ch(L) \mathrm{td}(Y).
\]
 Since $\pi^*:H^2(X,\Q)\to H^2(Y,\Q)$ is an injection of Hodge structures, it follows that 
\[\chi(L)=\sum_k\phi_k( \mathrm{c}_1(L)^k)\]
for Hodge classes $\phi_k\in\Sym^kH^2(X,\Q)^\vee$.  Moreover, from the existence of a simultaneous resolution $\mathscr{Y}\to\mathscr{X}$ of the universal locally trivial deformation $\mathscr{X}$ of $X$, it follows that the $\phi_k$ are locally constant and of type $(-k,-k)$ everywhere.  Now apply the proposition.
\end{proof}

For a compact complex space $W$ of dimension $k$, we denote by $[W]\in H_{2k}(W,\Z)$ the \emph{cycle class}, that is, the sum over the fundamental classes of the irreducible components of dimension $k$ weighted by their multiplicities. We write $\int_{W}:H^{2k}(W,\Z) \to \Z$ for the cap product with the cycle class. Hodge classes as in Proposition \ref{proposition fujiki relations} can be constructed via the following lemma. 

\begin{lemma}  Let $X$ be a primitive symplectic variety and $f:\mathscr{X}\to S$ a locally trivial deformation.  Let $\mathscr{W}\subset \mathscr{X}$ be a closed subvariety that is flat over $S$ with fiberwise dimension $k$.  Then $\int_{\mathscr{W}_s}$ defines a section of $\Sym^k R^2f_*\Q^\vee$ which is of type $(-k,-k)$.
\end{lemma}
\begin{proof} It suffices to show that for any sufficiently small Euclidean open set $U\subset S$, the cycle class $[\mathscr{W}_u]$ is constant in Borel--Moore homology $H^\BM_{2k}(f^{-1}(U),\Q)$. This is done as in \cite[Lemma~19.1.3]{Ful98}.
\end{proof}

For the following corollary, the term \emph{very general} is to be interpreted in terms of locally trivial deformations, i.e., outside a countable union of proper subvarieties in the base of the locally trivial Kuranishi family.

\begin{corollary}\label{corollary no odd subvarieties} Let $X$ be a very general primitive symplectic variety. Then $X$ does not contain odd dimensional closed subvarieties.
\end{corollary}
\begin{proof}  By the lemma, for a $k$-dimensional subvariety $W$ we have a Hodge class $\phi=\int_W$ in $\Sym^kH^2(X,\Q)^\vee$.  By taking a K\"ahler class $\omega\in H^2(X,\R)$, we see that $\int_W\omega^k>0$ and thus $\phi$ is nonzero, a contradiction.  \end{proof}

%------------------------------------------------------------------------------------------
\subsection{\texorpdfstring{$\Q$}{Q}-factoriality and \texorpdfstring{$\Q$}{Q}-factorial terminalizations}\label{sect Qfact}
%------------------------------------------------------------------------------------------
We first deduce the invariance of $\Q$-factoriality under locally trivial deformations for primitive symplectic varieties.

\begin{lemma}\label{lemma deformation qfactorial}
Let $X$ be a primitive symplectic variety. Then every small locally trivial deformation of $X$ is $\Q$-factorial if and only if $X$ is $\Q$-factorial.
\end{lemma}
\begin{proof}
Let $\pi:Y\to X$ be a resolution and consider $H^2(X,\Q)\subset H^2(Y,\Q)$ via pullback.  
Using Lemma~\ref{lemma simultaneous resolution}, we choose a simultaneous resolution $\scrY\to\scrX$ of the universal locally trivial deformation $\scrX\to\Def^\lt(X)$. Recall that by Proposition~\ref{proposition local torelli}, we can think of $\Def^\lt(X)$ as an open subset of the local period domain $\Omega(X)$.

For an element $\lambda\in H^2(X,\Q)$ with $q_X(\lambda)> 0$, let $T_\lambda\subset \Omega(X)$ be the locus for which $\lambda\in H^{2,0}(X)\oplus H^{0,2}(X)$.  Note that $T_\lambda$ is a totally real half-dimensional closed subvariety of $\Omega(X)$ (see Section~\ref{subsection orbits}).  We first claim that we may choose $\lambda$ so that $T_\lambda$ meets the image of $\Def^\lt(X)$.  Indeed, note that $q_X(\sigma)=0$ is equivalent to $q_X(\mathrm{Re}(\sigma))=q_X(\mathrm{Im}(\sigma))$ and $q_X(\mathrm{Re}(\sigma),\mathrm{Im}(\sigma))=0$.  Thus, taking $\lambda$ to be a rational class sufficiently close to $\mathrm{Re}(\sigma)$, then taking $R=\lambda$ and $I$ to be the projection of $\mathrm{Im}(\sigma)$ to $R^\perp$ scaled so that $q_X(R)=q_X(I)$, we can make $\sigma'=R+iI\in T_\lambda$ arbitrarily close to $\sigma$.

Now, choosing such a $\lambda$, in the notation of Lemma~\ref{lemma bbf topological} we have that $v(\lambda)\neq0$ by Proposition \ref{proposition fujiki relations}. Observe that the Fujiki constant is nonzero since $q_X(\sigma+\bar\sigma)\neq 0\neq v(\sigma+\bar\sigma)$.  Define a quadratic form $Q_\lambda$ on $H^2(Y,\Q)$ by the right-hand side of \eqref{eq integral form} divided by $q_{Y,\sigma}(\lambda)=q_{X,\sigma}(\lambda)$.  Note that:
\begin{enumerate}
\item\label{first part} $Q_\lambda$ is rational;
\item\label{second part} $Q_\lambda$ restricts to (a nonzero multiple of) $q_X$ on $H^2(X,\Q)$;
\item\label{third part} If $\lambda\in H^{2,0}(X)\oplus H^{0,2}(X)$ then
\[v(\lambda)^2q_{Y,\sigma}(\alpha) = q_{X,\sigma}(\lambda)Q_\lambda(\alpha)\]
for all $\alpha\in H^2(Y,\C)$, as in \cite[Théorème 5 (c)]{Bea83}.
\end{enumerate}
We now claim that $Q_\lambda$ is a morphism of Hodge structures. For this, we consider $Q_\lambda$ as a quadratic from on the local system of weight two Hodge structures associated to $\scrY \to \Def^\lt(X) \subset \Omega(X)$. In view of \eqref{first part}, it suffices to show that $Q_\lambda$ is a morphism of $\R$-Hodge structures. By \eqref{third part} and Lemma~\ref{lemma signature bbf form}, this is the case for all periods in $T_\lambda\cap \Def^\lt(X)$ and this set is nonempty and open in $T_\lambda$ by the above. But the Hodge locus of $Q_\lambda$ is certainly an analytic subset of $\Def^\lt(X)$ and therefore must be all of $\Def^\lt(X)$ as $T_\lambda$ is totally real and $\dim_\R T_\lambda=\dim_\C\Def^\lt(X)$.

Now, $q_X$ is nondegenerate by Lemma \ref{lemma signature bbf form}, so by \eqref{second part} the $Q_\lambda$-orthogonal space 
$$H^2(X,\Q)^\perp\subset H^2(Y,\Q)$$
 is a rational complement to $H^2(X,\Q)$ and is Hodge--Tate.  Thus, condition \eqref{eq weaker condition} is equivalent to \eqref{eq qfactoriality characterization}.  Using Lemma~\ref{lemma simultaneous resolution} again, we see that the validity of \eqref{eq qfactoriality characterization} is clearly invariant under locally trivial deformations.  We therefore conclude by Proposition~\ref{proposition analytic qfactoriality}.
\end{proof}
%------------------------------------------------------------------------------------------

The rest of this section will be devoted to relating the locally trivial deformation theory of a projective primitive symplectic variety $X$ to that of a $\Q$-factorial terminalization, which will play a role in the proof of surjectivity of the period map. We start with the following slight generalization of \cite[Lemma 3.5]{BL16}. The proof is literally the same as in loc. cit. so we omit it here.

%------------------------------------------------------------------------------------------
\begin{lemma}\label{lemma symplektisch}
Let $\pi:Y \to X$ be a proper bimeromorphic morphism between primitive symplectic varieties. Then $\pi^*:H^2(X,\C) \to H^2(Y,\C)$ is injective and the restriction of $q_Y$ to $H^2(X,\C)$ is equal to $q_X$. We have an orthogonal decomposition
\begin{equation}\label{eq n}
H^2(Y,\Q)=\pi^* H^2(X,\Q)\oplus N_\Q 
\end{equation}
where $N:=\tilde q^{-1}_{Y}(N_1(Y/X))$, which is negative definite.\qed
\end{lemma}

%------------------------------------------------------------------------------------------
Let $X$, $Y$ be normal compact complex varieties with rational singularities and let $\pi:Y \to X$ be a proper bimeromorphic morphism. It follows that $\pi_*\sO_Y = \sO_X$ and $R^1\pi_*\sO_Y=0$ so that 
by \cite[Proposition 11.4]{KM92}, there is a commutative diagram 
\begin{equation}
\label{eq defo diag}
\xymatrix{
\scrY \ar[d]\ar[r]^{P}& \scrX \ar[d]\\
\Def(Y) \ar[r]^{p} & \Def(X) \\
}
\end{equation}
for the miniversal families of deformations of $X$ and $Y$. Let us consider the case that $\pi:Y\to X$ is a $\Q$-factorial terminalization of a projective primitive symplectic variety. We will show below (Proposition \ref{prop defo}) that the locally trivial deformations of $X$ are identified via $p$ with the locus of deformations of $Y$ where the classes of contracted curves remain Hodge.

\begin{proposition}\label{prop defo}
Let $X$ and $Y$ be projective primitive symplectic varieties and let $\pi: Y\to X$ be a proper bimeromorphic morphism. Assume $Y$ is $\Q$-factorial and terminal. Let $N\subset H^2(Y,\C)$ be the $q_Y$-orthogonal complement to $H^2(X,\C) \subset H^2(Y,\C)$ and consider the diagram \eqref{eq defo diag}. Denote by $\Def(Y,N)\subset \Def(Y)$ the subspace of deformations such that classes in $N$ remain of type $(1,1)$. Then the following holds:
\begin{enumerate}
	\item $p^{-1}(\Def^\lt(X)) = \Def(Y,N)\subset \Def(Y)$.
	\item The restriction $p:\Def(Y,N)\to\Def^\lt(X)$ is an isomorphism.
\end{enumerate}
\end{proposition}
%------------------------------------------------------------------------------------------
\begin{proof}
By Theorem \ref{theorem deflt is smooth} respectively \cite[Main~Theorem]{Nam06}, the spaces $\Def^\lt(X)$ and $\Def(Y)$ are smooth of dimension $h^{1,1}(X)$ and $h^{1,1}(Y)$, respectively. Moreover, by \cite[Theorem 1]{Nam06}, $\Def(X)$ is smooth while $p:\Def(Y)\to\Def(X)$ is finite and, as both are of the same dimension, surjective.

Now, $\Def(Y,N)\subset \Def(Y)$ is a smooth subvariety of codimension $m:=\dim N$ whose tangent space is identified with $H^{1,1}(X)$ under the period map, see Lemma~\ref{lemma deformations line bundle}. By Corollary~\ref{cor deformation symplectic}, the fibers of the universal deformations $\scrY \to \Def(Y)$ and $\scrX \to \Def^\lt(X)$ are primitive symplectic varieties. Therefore, \cite[Lemma 2.2]{BL16} entails that the second cohomology of locally trivial deformations of $X$ form a vector bundle on $\Def^\lt(X)$, in particular, $h^{1,1}(\scrX_{p(t)})=h^{1,1}(X)$. Thus, by the decomposition $H^2(Y,\C)=N\oplus H^2(X,\C)$ from Lemma \ref{lemma symplektisch} we see that the space $N_1(\scrY_t/\scrX_{p(t)})$ of curves contracted by $P_t:\scrY_t \to \scrX_t$ has dimension $m$ for all $t\in p^{-1}(\Def^\lt(X))$. As $N$ is the orthogonal complement of $H^2(X,\C)$, it also  varies in a local system. Using the period map this shows that $p^{-1}(\Def^\lt(X)) = \Def(Y,N)$. 

One shows as in \cite[Proposition 2.3 (ii)]{LP16} that $p$ is an isomorphism, see also \cite[Proposition 4.5]{BL16}.
\end{proof}

We will need the following corollary in Section~\ref{section torelli}. For a projective primitive symplectic variety $X$ and a $\Q$-factorial terminalization $\pi:Y\to X$, let $g:\scrY\to\Def(Y)$ and $f:\scrX\to\Def(X)$ be the universal deformations, and let $f':\scrX'\to\Def(Y)$ be the pullback of $\scrX$ to $\Def(Y)$ along $p$ as in \eqref{eq defo diag}.  Then $P':\scrY\to\scrX'$ is a simultaneous $\Q$-factorial terminalization by \cite[Main~Theorem]{Nam06} and Lemma~\ref{lemma deformation qfactorial}.  Consider the constant second Betti number locus
\[B_X:=\{t\in \Def(Y)\mid \rk (R^2f'_*\Q_{\scrX'})_t=b_2(X)\}\]
which is a (reduced) closed analytic subspace of $\Def(Y)$.
\begin{corollary}\label{corollary betti locus}In the above setup, $B_X=\Def(Y,N)$.
\end{corollary}
\begin{proof}
Certainly $B_X\supset \Def(Y,N)$ by the proposition and \cite[Lemma 2.4]{BL16}.  By Lemma \ref{lemma symplektisch} and proper base change we have an injection 
\[0\to (R^2f'_*\Q_{\scrX'})_t\to (R^2g_*\Q_{\scrY})_t\]
for all $t\in \Def(Y)$.  The restrictions $(R^2g_*\Q_{\scrX'})|_{B_X}$ and $(R^2f'_*\Q_{\scrY})|_{B_X}$ are local systems as therefore is the orthogonal $(R^2f'_*\Q_{\scrX'})|_{B_X}^\perp$ in $(R^2f'_*\Q_{\scrY})|_{B_X}$.  We must then have $(R^2f'_*\Q_{\scrX'})_t^\perp=N$ for all $t\in B_X$, but since $(R^2f'_*\Q_{\scrX'})_t^\perp$ is Hodge--Tate we obtain the reverse inclusion $B_X\subset \Def(Y,N)$.
\end{proof}

%------------------------------------------------------------------------------------------
%------------------------------------------------------------------------------------------
%------------------------------------------------------------------------------------------
\section{The projectivity criterion}\label{section projectivity criterion}
%------------------------------------------------------------------------------------------
%------------------------------------------------------------------------------------------
%------------------------------------------------------------------------------------------

In this section we formulate and prove an analog of Huybrechts' projectivity criterion \cite[Theorem 3.11]{Huy99} (see also \cite{Huy03erratum}) in the singular setup. Note that for orbifold singularities, the question has been examined by Menet \cite{Men20}. We use several of his as well as of Huybrechts' arguments.

%------------------------------------------------------------------------------------------
\subsection{A singular version of the Demailly--P\u aun Theorem}\label{section demailly paun}
%------------------------------------------------------------------------------------------

We do not know whether the analog of Demailly--P\u aun's celebrated theorem \cite[Main Theorem~0.1]{DP04} on the numerical characterization of the Kähler cone of a compact Kähler manifold holds for singular varieties. One may however easily deduce from it that a similar statement holds which is good enough for our purposes. For this purpose, we introduce a notion for cohomology classes that behave as if they were Kähler classes. 

Recall from \eqref{eq real hodge classes} that we defined $H^{1,1}(X,\R)=F^1H^2(X,\C)\cap H^2(X,\R)$ for a reduced compact complex space of class $\scrC$.

\begin{definition}\label{definition demailly paun}
Let $X$ be a reduced compact complex space of class $\scrC$ and consider a class $\kappa \in H^{1,1}(X,\R)$. We say that $\kappa$ is \emph{Demailly--P\u aun} if for every compact complex manifold $V$ and for every generically finite morphism $f:V \to X$ the class $f^*\kappa$ is big and nef. We denote by $\DP(X)\subset H^{1,1}(X,\R)$ the convex cone consisting of all Demailly--P\u aun classes. We refer to it as the \emph{Demailly--P\u aun cone}.
\end{definition}

This definition desserves a couple of comments.

\begin{remark}\label{remark demailly paun}\
\begin{enumerate}
\item Every Kähler class is Demailly--P\u aun, in particular, $\DP(X) \neq \emptyset$ if $X$ is Kähler. Indeed, every Kähler class is a $(1,1)$-class by Proposition~\ref{proposition kaehler is 11}. Then the claim follows as the pullback of a Kähler class under a generically finite morphism from a smooth variety is big and nef.
	\item We do not know of an example of a class that is Demailly--P\u aun but not Kähler. It seems likely that Demailly--P\u aun classes are the same as Kähler classes. Apart from the Demailly--P\u aun theorem \cite[Main Theorem~0.1]{DP04}, evidence for this presumption is given in \eqref*{remark demailly paun item three}.
	\item\label{remark demailly paun item three} Every rational Demailly--P\u aun class is Kähler. Indeed, a multiple of such a class is the first Chern class of a big line bundle $L$. Therefore, $X$ is Moishezon and $L$ is ample by the Nakai--Moishezon criterion. Note that the Nakai--Moishezon criterion holds for big line bundles on Moishezon varieties, see e.g. \cite[3.11~Theorem]{Kol90}.
\item A closed subvariety of a class $\scrC$ variety is again dominated by a compact Kähler manifold, see Proposition~\ref{proposition kaehler space}, and so it is itself class $\scrC$. Then it is immediate that the restriction of a Demailly--P\u aun class to a subvariety is again Demailly--P\u aun.
\item The assumption that $X$ be of class $\scrC$ is somewhat redundant but simplifies the exposition. If for some $\kappa \in H^{1,1}_\ddbar(X,\R)$ the pullback $\pi^*\kappa$ along a resolution $\pi:Y \to X$ is big, then $Y$ (and hence also $X$) are of class $\scrC$ by \cite[Theorem~0.7]{DP04}.
\end{enumerate}
\end{remark}

\begin{lemma}\label{lemma demailly paun}
Let $X$ be a compact variety of class $\scrC$ and let $\kappa \in H^{1,1}(X,\R)$. Then $\kappa$ is Demailly--P\u aun if and only if for every compact complex manifold $W$ and for every holomorphic map $\pi:W\to X$ which is bimeromorphic onto its image the class $\pi^*\kappa$ is big and nef. Moreover, the pullback of a Demailly--P\u aun class to an arbitrary compact complex manifold is nef.
\end{lemma}
\begin{proof}
To prove the non-trivial direction of the first claim, let $\pi:V \to X$ be a holomorphic map from a compact complex manifold which is generically finite onto its image. We denote $\bar V:= \pi(V)$ and factor $\pi$ as $V \to[\pi_1] \bar V \to[\pi_2] X$. We then chose a diagram
\[
 \xymatrix{
 W_2 \ar[r]^\phi \ar[d]_f & V \ar[dr]^{\pi} \ar[d]_{\pi_1} &\\
 W_1 \ar[r]^\psi& \bar V \ar[r]^{\pi_2} & X\\
 }
\]
where $W_1, W_2$ are compact Kähler manifolds and $W_1 \to \bar V$, $W_2\to V$ are bimeromorphic. By assumption, $\alpha:=\psi^*\pi_2^*\kappa$ is big and nef. By a result of P\u aun \cite[Théorème~1]{Pau98}, nefness of $\alpha$ is equivalent to $f^*\alpha$ being nef. Bigness is preserved under generically finite pullbacks so that $f^*\alpha$ is big and nef. Since $W_2 \to V$ is bimeromorphic between compact complex manifolds, $\pi^*\kappa$ is big and nef as $\phi^*\pi^*\kappa=f^*\alpha$ is.

For the second statement, let $\pi:V \to X$ be a morphism from a compact complex manifold. We change the above diagram accordingly and deduce the claim by invoking P\u aun's result once more.
\end{proof}

The main result of this section is deduced from the smooth Demailly--P\u aun theorem and P\u aun's results in \cite{Pau98} via an inductive argument. Note that while our result is not essentially new compared to the Demailly--P\u aun theorem, it should be mentioned that Collins--Tosatti proved a true generalization of the Demailly--P\u aun theorem in \cite[Theorem 1.1]{CT16} for possibly singular compact subvarieties of Kähler manifolds. 

\begin{theorem}\label{theorem demailly paun}
Let $X$ be a reduced compact complex space of class $\scrC$ and consider the cone $P\subset H^{1,1}(X,\R)$ of all classes $\alpha$ on $X$ such that for all closed analytic subsets $V \subset X$  we have
\[
\int_V\alpha^{\dim V} > 0.
\]
Then the Demailly--P\u aun cone $\DP(X)$ is empty or a connected component of $P$. If $X$ is Kähler, $\DP(X)$ is the connected component of $P$ containing the Kähler cone. 
\end{theorem}
\begin{proof}
Clearly, $\DP(X)\subset P$ and as the Demailly--P\u aun cone is convex, it is contained in a connected component of $P$. Moreover, if $X$ is Kähler, the Kähler cone is contained in $\DP(X)$.

For the converse, we may assume that $\DP(X)$ is non-empty, otherwise there is nothing to prove. Let $\alpha\in P$ be a class in the same connected component as $\DP(X)$. We will prove that the restriction of $\alpha$ to any subvariety of $X$ is Demailly--P\u aun by induction on the dimension of the subvariety.

For $d=0$ the statement is trivial. Let $V\subset X$ be a subvariety of dimension $d$ and assume that $\alpha$ is Demailly--P\u aun on every subvariety of $X$ of dimension strictly smaller than $d$.
 We denote by $\pi:W\to X$ the composition of a resolution of singularities of $V$ with the inclusion $V \subset X$ where $W$ is a compact Kähler manifold of dimenion $d$. Such a resolution exists thanks to Proposition \ref{proposition kaehler space}. 
 By Lemma \ref{lemma demailly paun} it suffices to prove that $\pi^*\alpha$ is big and nef. Clearly, $\alpha\vert_V$ fulfills the hypotheses of the theorem if $\alpha$ does.  We show first that $\pi^*\alpha$ is nef on $W$ using the Demailly--P\u aun theorem on $W$. Let us take a Kähler class $\kappa$ on $W$. For $0< \veps \ll 1$ the class $\alpha_W:=\pi^*\alpha + \veps \kappa$ satisfies $\alpha_W^d > 0$. If $Z \subset W$ is a proper analytic subvariety of dimension $e$, then $\pi(Z) \subset V$ is also a proper subvariety and thus $\alpha\vert_{\pi(Z)}$ is Demailly--P\u aun by the inductive hypothesis. 
We will show that $\int_Z \alpha_W^{e}>0$. But this can be computed on a resolution of singularities, so we may without loss of generality assume that $Z$ is nonsingular. Then $\pi^*\alpha\vert_Z$ is nef by Lemma \ref{lemma demailly paun} and therefore $\alpha_W\vert_Z$ has positive top self intersection.

As $\alpha$ is in the same connected component of $P=P(V)$ as the Demailly--P\u aun classes on $V$, also $\alpha_W$ is in the same connected component $P(W)$ as the Demailly--P\u aun classes on $W$. But by \cite[Main Theorem~0.1]{DP04}, we have $\DP(W)=\sK(W)$ where $\sK(W)$ denotes the Kähler cone.
Hence, the Demailly--P\u aun theorem applies and $\alpha_W$ is Kähler. Moreover, $\pi^*\alpha$ is nef on $W$ because $\veps$ was arbitrarily small.
But certainly $\int_W(\pi^*\alpha)^d > 0$ and therefore $\pi^*\alpha$ is also big on $W$ by \cite[0.4 Theorem]{DP04}. This concludes the proof.
\end{proof}

\subsection{Projectivity criterion}

In this section, the term \emph{very general} is to be interpreted in terms of locally trivial deformations, i.e., outside a countable union of proper subvarieties in the base of the locally trivial Kuranishi family.
\begin{definition}\label{definition positive cone}
Let $X$ be a primitive symplectic variety and $q_X$ its BBF form. We define the positive cone
\begin{equation}
\sC_X:=\left\{ \alpha\in H^{1,1}(X,\R)\middle| \ q_X(\alpha)>0 \right\}^\kappa
\end{equation}
where $\kappa$ denotes the connected component containing the Kähler cone.
\end{definition}

\begin{theorem}\label{theorem demailly paun cone}
For a very general primitive symplectic variety $X$, the positive cone equals the Demailly--P\u aun cone:
\begin{equation}\label{eq dpx gleich cx}
\DP(X) = \sC_X.
\end{equation}
\end{theorem}
\begin{proof}
The Demailly--P\u aun cone is always contained in the positive cone by Theorem~\ref{theorem demailly paun}. Let us prove the other inclusion. By Corollary \ref{corollary no odd subvarieties}, $X$ does not contain any odd dimensional subvarieties. Let $Z\subset X$ be a subvariety and denote by $2d$ its dimension. Choose a Kähler class $\kappa$ on $X$. 
Then by the Fujiki relations, Proposition~\ref{proposition fujiki relations}, there is a constant $c_Z \in \R$ such that for every $\alpha\in H^2(X,\C)$ the equality 
\[
\begin{aligned}
c_Z \cdot \int_Z(\alpha S + \kappa T)^{2d} &= q_X(\alpha S+ \kappa T)^d \\
& = \left( q_X(\alpha) S^2 +2q_X(\alpha,\kappa) ST+ q_X(\kappa)T^2\right)^d
\end{aligned}
\]
of polynomials in the indeterminates $S$ and $T$ holds. Choosing $\alpha=\kappa$ we see that $c_Z$ has to be strictly positive. From now on let $\alpha \in \sC_X$. As also $\kappa \in\sC_X$, Lemma \ref{lemma signature bbf form} implies that $q_X(\alpha,\kappa)> 0$. The coefficients of the polynomial on the right-hand side are manifestly all positive. We conclude from looking at the left-hand side that for every $0\leq \lambda \leq 1$ we have that $\lambda \alpha + (1-\lambda) \kappa$ lies in the cone $P$ from Theorem \ref{theorem demailly paun}. In particular, $\alpha$ is in the connected component of $P$ containing the Kähler cone $\sK(X)$. We conclude from Theorem \ref{theorem demailly paun} that $\alpha \in \DP(X)$.
\end{proof}

The following is the singular version of \cite[Theorem 3.11]{Huy03cone} and the proof relies on important ideas of his and of Menet \cite{Men20}, see Section 4 of Menet's article. The presentation follows \cite[Proposition 26.13]{GHJ03}.

\begin{theorem}\label{theorem projectivity criterion} Let $X$ be a primitive symplectic variety and $\alpha \in H^2(X,\Z)$ a $(1,1)$-class. If $q(\alpha)> 0$, then $X$ is projective.
\end{theorem}
Note that the existence of such a class can be read off only from the period. 

\begin{proof} 
By the Lefschetz $(1,1)$-theorem, there is a line bundle $L$ on $X$ with first Chern class $\chern_1(L)=\alpha$. We will show that $L$ is big. It suffices to do this on a resolution, say $\pi:Y\to X$, as bigness of a line bundle is a birationally invariant notion. Bigness of the line bundle $\pi^*L$ is implied by bigness of $\pi^*\alpha$, see \cite[Theorem~4.6]{JS93}. The strategy is to infer bigness of $\alpha$ by approximating $\alpha$ on a resolution with Kähler currents on nearby varieties. 

Let us consider the locally trivial Kuranishi family $\sX\to S:=\Def^\lt(X)$ and take a simultaneous resolution $\sY\to\sX$ which is possible by Lemma~\ref{lemma simultaneous resolution}. From now on we choose $\pi:Y\to X$ to be the special fiber of $\sY \to \sX$.  For a very general $t\in S$ the corresponding primitive symplectic varieties $\sX_t$ satisfy $\DP(\sX_t)=\sC_{\sX_t}$ thanks to Theorem \ref{theorem demailly paun cone}. Therefore, $\alpha$ can be approximated by Demailly--P\u aun classes $\alpha_{t_i}$ on $\sX_{t_i}$ where $t_i \to 0\in S$ for $i\to \infty$, where $X$ is the fiber of $\sX \to S$ over $0$. Consequently, $\pi^*\alpha$ can be approximated by big classes on nearby fibers $\sY_{t_i}$ and as in \cite[Proposition 6.1]{Dem92}, see also the proof of \cite[Proposition 26.13]{GHJ03}, we deduce that $\pi^*\alpha$ is big. The key point here is to see that in the above approximation procedure, the limit of a sequence of closed positive currents are again closed and positive. This is explained in detail in the appendix by Diverio to \cite{ADH19}. As explained before, bigness of $\pi^*\alpha$ implies that $\pi^*L$ and hence $L$ is big. Thus, $X$ is Moishezon. Being Kähler and having rational singularities, it must be projective by \cite[Theorem~1.6]{Nam02}.
\end{proof}

The following result is the singular analog of \cite[Theorem 4.8 2)]{Fuj83}, see also \cite[Theorem 3.5]{Huy99} and \cite[Proposition 26.6]{GHJ03}. We have to change the proof slightly in the singular setting.

\begin{corollary}\label{corollary density of projectives} Let $X$ be a primitive symplectic variety, $f:\sX\to \Def^\lt(X)$ be the universal locally trivial deformation of $X=f^{-1}(0)$, and $S \subset \Def^\lt(X)$ a positive-dimensional subvariety through $0 \in \Def^\lt(X)$. Then in every open neighborhood $U\subset S$ of $0$ there is a point $s\in U$ such that the fiber $\sX_s$ is projective.
\end{corollary}
\begin{proof}
The proof is almost the same as in \cite[Theorem 4.8 2)]{Fuj83} respectively \cite[Theorem 3.5]{Huy99}. We refer to these references for details and content ourselves with a sketch of proof. One restricts to a one-dimensional disk $S\subset \Def^\lt(X)$ and chooses a Kähler form $\omega$ on $X$ such that the locus $S_{[\omega]}\subset \Def^\lt(X)$ where the class $[\omega]$ remains of type $(1,1)$ intersects $S$ transversally. Next one chooses classes $\alpha_i \in H^2(X,\Q)$ converging to $[\omega]$ such that the $\alpha_i$ are not of type $(1,1)$ on $X$. Then the $(1,1)$-locus $S_{\alpha_i}\subset \Def^\lt(X)$ intersects $S$ in points $t_i\neq 0$ converging to $0$. Now the idea is that the $(1,1)$-class $\alpha_i$ is Kähler on $\sX_{t_i}$ for $t_i$ sufficiently close to $0$. In \cite[Theorem 3.5]{Huy99} this is seen via harmonic representatives. As $X$ is singular, we cannot argue literally the same. However, due to Lemma \ref{lemma simultaneous resolution} we may take a simultaneous resolution $\pi:\sY\to\sX$ obtained by successive blow ups. In particular, there is an $\R$-linear combination $E$ of exceptional divisors such that for $e:=\chern_1\left(\sO(E)\right)$ we have that $\alpha_i-e$ is Kähler on $Y:=\pi^{-1}(X)$. Now we apply the argument involving harmonic representatives to $\alpha_i- e$ and deduce that for $t_i$ sufficiently close to $0$ the variety $\sY_{t_i}$ is projective. Hence, also the corresponding $\sX_{t_i}$ is projective by \cite[Theorem 1.6]{Nam02}.
\end{proof}

We immediately deduce

\begin{corollary}\label{corollary density of projectives2} Let $X$ be a primitive symplectic variety and let $f:\sX\to \Def^\lt(X)$ be the universal locally trivial deformation of $X=f^{-1}(0)$. Then for every positive-dimensional subvariety $S \subset \Def^\lt(X)$ the set of points $\Sigma\subset S$ with projective fiber is dense.\qed
\end{corollary}

%------------------------------------------------------------------------------------------
\subsection{Inseparability and moduli}\label{subsection inseparable}
%------------------------------------------------------------------------------------------
Given a primitive symplectic variety $X$ and a lattice $\Lambda$ with quadratic form $q$, a \emph{$\Lambda$-marking} of $X$ is an isomorphism $$\mu : (H^2(X,\Z)_\tf,q_X)\xrightarrow{\cong}(\Lambda,q).$$  A \emph{$\Lambda$-marked primitive symplectic variety} is a pair $(X,\mu)$ where $X$ is a primitive symplectic variety and $\mu$ is a $\Lambda$-marking of $X$. Two $\Lambda$-marked primitive symplectic varieties $(X,\mu)$ and $(X',\mu')$ are \emph{isomorphic} if there is an isomorphism $\vphi:X \to X'$ such that $\mu'=\mu \circ \vphi^*$.

\begin{definition}\label{definition marked moduli}
Given a lattice $\Lambda$ as above, we denote by $\gothM_\Lambda$ the analytic coarse moduli space of \emph{$\Lambda$-marked} primitive symplectic varieties.  As a set, $\gothM_\Lambda$ consists of isomorphism classes of $\Lambda$-marked primitive symplectic varieties $(X,\mu)$, and it is given the structure of a not-necessarily-Hausdorff complex manifold using Theorem \ref{theorem deflt is smooth} by identifying points in the bases of locally trivial Kuranishi families over which the fibers are isomorphic as $\Lambda$-marked varieties.
\end{definition}

Note that this definition coincides with the usual one \cite[1.18]{Huy99} for irreducible symplectic manifolds due to the fact that all deformations of smooth varieties are locally trivial. The following statement of Huybrechts' carries over together with its proof.
\begin{theorem}\label{theorem inseparable implies birational}
Let $X$, $X'$ be primitive symplectic varieties such that for some choice of marking $\mu:H^2(X,\Z)_\tf \to \Lambda$, $\mu':H^2(X',\Z)_\tf \to \Lambda$ the pairs $(X,\mu)$, $(X',\mu')$ define non-separated points in the $\Lambda$-marked moduli space. Then there is a bimeromorphic map $\phi:X \ratl X'$.
\end{theorem}
\begin{proof}
Identical to \cite[Theorem 4.3]{Huy99} using a simultaneous resolution.
\end{proof}

\begin{corollary}\label{cor bij mt general}If $(X,\mu)$ and $(X',\mu')$ are inseparable in moduli with Mumford--Tate general periods, then $(X,\mu)=(X',\mu')$.
\end{corollary}
\begin{proof}By the theorem, there is a bimeromorphic $\phi:X\ratl X'$. Mumford-Tate generality implies that neither $X$ nor $X'$ contain compact curves. Indeed, such a curve would define a non-zero Hodge class e.g. in $H_2(X,\Q)$, so by the BBF form we also had a non-zero Hodge class in $H^2(X,\Q)$. By a standard argument, bimeromorphic maps between normal varieties without curves are necessarily isomorphisms.\footnote{This can be seen exactly as for projective algebraic varieties of Picard rank one by applying e.g. \cite[Lemma 1.15 (b)]{Deb01} to a resolution of indeterminacies and its inverse.} We therefore obtain an isomorphism of Hodge structures $H^2(X,\Z)_\tf\to H^2(X',\Z)_\tf$ which maps a K\"ahler class to a K\"ahler class.  The claim follows since the automorphism group of a Mumford--Tate general period $\Aut_{\Hdg}(H^2(X,\Z)_\tf)=\{\pm 1\}$, since
\[
\End(H^2(X,\Q),q_X)^{\SO(H^2(X,\Q),q_X)}=\Q\,\id.
\]
\end{proof}
%------------------------------------------------------------------------------------------
We denote by $\Delta=\{z\in \C\mid \abs{z}\leq 1\}$ the complex unit disk and by $\Delta^*:=\Delta\ohnenull$ the complement of the origin. Recall that if two not necessarily $\Q$-factorial complex varieties are bimeromorphic, it is not in general true that we can push forward (or pull back) line bundles along the bimeromorphic map.
\begin{theorem}\label{theorem huybrechts strong}
Let $X$ and $X'$ be projective primitive symplectic varieties, and
let $\phi:X\ratl X'$ be a birational map which is an isomorphism in codimension $1$ such that $\phi_*: \Pic(X)_\Q \to \Pic(X')_\Q$ is well-defined and an isomorphism. Then there are one parameter locally trivial deformations $f:\scrX \to \Delta$, $f':\scrX' \to \Delta$ such that $\scrX$ and $\scrX'$ are birational over $\Delta$ and such that $\scrX^* = f^{-1}(\Delta^*) \isom (f')^{-1}(\Delta^*) = (\scrX')^* $. 
\end{theorem}
\begin{proof}
The basic strategy of \cite[Theorem 4.6]{Huy99} remains unchanged, we will therefore only explain where we need to deviate from it. By Corollary \ref{corollary riemann roch via q}, there are polynomials $f_X(t)$ and $f_{X'}(t)$ with rational coefficients of degree $n=\dfrac{\dim X}{2}$ such that for any line bundle $L$ on $X$ a Hirzebruch--Riemann--Roch statement of the form $\chi(X,L)=f_X(q_X(\chern_1(L)))$ holds and similarly for $X'$. We may assume that $f_X\geq f_{X'}$ with respect to the lexicographic order and choose an ample line bundle $L'$ on $X'$ and denote by $L$ the corresponding $\Q$-line bundle on $X$. Replacing $L'$ by a multiple, we may assume that $L$ is integral. Let $\pi:(\scrX,\scrL)\to S$ be a locally trivial deformation of $(X,L)$ over a smooth one-dimensional base such that the Picard number of the general fiber of $\scrX\to S$ is one.  As in \cite[Theorem 4.6]{Huy99}, using the projectivity criterion from Theorem \ref{theorem projectivity criterion} one shows that $h^0(\scrL_t^{\tensor m})$ for $m\gg 0$ does not depend on $t\in S$ that the associated linear system gives a meromorphic $S$-morphism $\scrX\ratl \P_S(\pi_*\scrL^\vee)$ which is bimeromorphic onto its image. We obtain $\scrX'\to S$ as the closure of this image and one verifies as in \cite[Proposition 4.2]{Huy97} that $\scrX'\to S$ has the desired properties, in particular, that its central fiber is $X'$.
\end{proof}

This result can be reformulated as follows.
%------------------------------------------------------------------------------------------
\begin{corollary}\label{corollary huybrechts strong}
Let $X$ and $X'$ be projective primitive symplectic varieties, and
let $\phi:X\ratl X'$ be a birational map which is an isomorphism in codimension $1$ such that $\phi_*: \Pic(X)_\Q \to \Pic(X')_\Q$ is well-defined and an isomorphism. Then for every choice of a marking $\mu:H^2(X,\Z)_\tf\to \Lambda$ there exists a marking $\mu':H^2(X',\Z)_\tf\to \Lambda$ such that the points $(X,\mu)$ and $(X',\mu')$ are inseparable points in the moduli space $\gothM_\Lambda$.
\end{corollary}
%------------------------------------------------------------------------------------------

%------------------------------------------------------------------------------------------
%------------------------------------------------------------------------------------------
%------------------------------------------------------------------------------------------
\section{Projective degenerations}\label{section proj degen}
%------------------------------------------------------------------------------------------
%------------------------------------------------------------------------------------------
%------------------------------------------------------------------------------------------
The main goal of this section  is to prove the following result, which will be needed for the surjectivity of the period map in Section~\ref{section torelli}:

\begin{theorem}\label{filling holes}  Let $f:\mathscr{X}^*\to \Delta^*$ be a projective locally trivial family of primitive symplectic varieties with $\Q$-factorial terminal singularities such that the monodromy of $R^2f_*\Q_{\scrX^*}$ is finite.  Then there is a proper locally trivial family $g:\mathscr{Y}\to \Delta $ of primitive symplectic varieties whose restriction $\mathscr{Y}|_{\Delta^*}\to \Delta^* $ is isomorphic to the restriction of the base-change of $\scrX^*\to\Delta^*$ along a finite \'etale cover $\Delta^*\to\Delta^*$. 
\end{theorem}

Theorem \ref{filling holes} is proven for smooth $\mathscr{X}^*\to \Delta^*$ in \cite[Theorem 1.7]{KLSV18}, and the proof in our slightly more general setting involves very mild modifications of the same arguments given Proposition \ref{sym}, albeit rearranged slightly and with some simplifications.

A crucial step is the following version of \cite[Theorem 2.1]{KLSV18} which uses the MMP to produce nice models for degenerations of $K$-trivial varieties.

\begin{theorem}(Theorem 2.1 and Remarks 2.3 and 2.4 of \cite{KLSV18})\label{thm mmp filling}  Let $f:\scrX\to\Delta$ be a projective family whose generic fiber is a $K$-trivial variety with $\Q$-factorial terminal singularities and such that at least one component of the special fiber is not uniruled.  Then there is a projective family $g:\scrY\to\Delta$ for which:
\begin{enumerate}
\item the restriction $\scrY|_{\Delta^*}\to\Delta^*$ is isomorphic to the restriction of the base-change of $\scrX\to\Delta$ along a finite cover $\Delta\to\Delta$;
\item the special fiber is a $K$-trivial variety with canonical singularities;
\item the total space $\scrY$ has terminal singularities.
\end{enumerate}
\end{theorem}

Note that the third statement follows from the proof in \cite{KLSV18}.  Theorem \ref{thm mmp filling} reduces the proof of Theorem \ref{filling holes} to showing that the assumption on the local monodromy implies that some component of a degeneration must be non-uniruled, and this is accomplished by the following:

\begin{proposition}\label{proposition sympl comp}  Let $f:\mathscr{X}\to \Delta$ be a flat projective family such that:
\begin{enumerate}

\item the restriction $\scrX^*:=\scrX|_{\Delta^*}\to\Delta^*$ is a locally trivial family of primitive symplectic varieties;
\item the local monodromy of $R^2f_*\Q_\scrX$ is trivial;
\item the special fiber $X$ has no multiple components;
\item the total space $\scrX$ has log terminal singularities.

\end{enumerate}
Then a resolution of some component of the special fiber $X$ has a generically nondegenerate holomorphic 2-form. 
\end{proposition}
\begin{proof}

Let $2n$ be the fiber dimension of $f$ and take $\pi:(\scrY,Y)\to(\scrX,X)$ to be a log resolution and $g:=f\circ \pi:\scrY\to\Delta$.  After possible shrinking $\Delta$, the morphism \mbox{$\pi:\scrY\to\scrX$} is a fiberwise resolution over $\Delta^*$.  Recall that there is a specialization map $\sp: H^*(Y,\Q)\to H^*(\mathscr{Y}_\infty,\Q)$ which is topologically constructed as follows.  After possibly shrinking $\Delta$ we let $\mathscr{Y}_\infty=e^*\mathscr{Y}_{\Delta^*}$, where $e:\mathbb{H}\to \Delta^*$ is the universal cover.  Then $\sp$ is the pullback along the natural map $\scrY_\infty\to\scrY$ composed with the isomorphism induced by the inclusion $Y\to\scrY$ which is a homotopy equivalence.  Note that $\sp$ is a ring homomorphism, and that the inclusion $\scrY_t\to \scrY_\infty$ of a fiber above $t\in\Delta^*$ is also a homotopy equivalence, as locally trivial families are topologically (even real analytically) trivial \cite[Proposition~5.1]{AV19}.  We can also view $H^*(\scrY_\infty,\Q)$ as the nearby cycles $\psi Rf_*\Q_{\scrY}$ (up to a shift) and the specialization map as the natural map $i^*Rf_*\Q_{\scrY}\to\psi Rf_*\Q_{\scrY}$ by proper base-change, where $i:\{0\}\to\Delta$ is the inclusion.  By Saito's theory \cite{Sai88,Sai90}, this is a morphism of mixed Hodge structures, the mixed Hodge structure on $\psi Rf_*\Q_\scrY$ being the limit mixed Hodge structure.

Now for $t\in\Delta^*$, the pullback $\pi_t^*:H^*(\scrX_t,\Q)\to H^*(\scrY_t,\Q)$ induces an injection $\mathrm{gr}^W_{k}H^k(\scrX_t,\Q)\to H^k(\scrY_t,\Q)$ for all $k$.  By Theorem \ref{sym}, for $k\leq n$ we have an induced injection $\Sym^kH^2(\scrX_t,\Q)\to H^{2k}(\scrY_t,\Q) $ and therefore also an injection $\Sym^kH^2(\scrX_\infty,\Q)\to H^{2k}(\scrY_\infty,\Q)$.

\begin{claim}  The image of the specialization $\sp:H^{2k}(Y,\Q)\to H^{2k}(\scrY_\infty,\Q)$ contains the image of $\Sym^kH^2(\scrX_\infty,\Q)$ for $k\leq n$.\end{claim}

\begin{proof} By the semisimplicity of the category of variations of polarized integral Hodge structures, $\Sym^k R^2f_*\Q_{\scrX^*}$ is a summand of $R^{2k}g_*\Q_{\scrY^*}$ for $k\leq n$.  By the decomposition theorem \cite[\S 5.3]{Sai88}, the intermediate extension $j_{!*}(\Sym^kR^{2}f_*\Q_{\scrX^*}[1])[-2k-1]$ is a summand of $Rg_*\Q_{\scrY}$, where $j:\Delta^*\to\Delta$ is the inclusion.  As the monodromy of $\Sym^kR^2f_*\Q_{\scrX^*}$ is trivial, the specialization map
\[i^*j_{!*}(\Sym^kR^{2}f_*\Q_{\scrX^*}[1])\to \psi j_{!*}(\Sym^kR^{2}f_*\Q_{\scrX^*}[1])\]
is an isomorphism, hence the claim.
\end{proof}

Now, $H^2(\scrY_\infty,\Q)$ has the same Hodge numbers as the general fiber (since the monodromy is trivial), and it follows that there is an element $w\in I^{2,0}H^2(Y,\Q)$ mapping to a generator of $I^{2,0}H^2(\scrY_\infty,\Q)$. Here, $I^{2,0}$ denotes the $(2,0)$-part of the Deligne splitting, see e.g. \cite[Lemma-Definition~3.4]{PS08}.  Moreover, $w^{n}\neq 0$ by the claim.  The same is true on the normalization $\tilde{Y}\to Y$, so some component of $Y$ has a generically nondegenerate holomorphic 2-form.  Finally, since $\scrX$ is log terminal, by \cite[Corollary~1.5]{HM07} the exceptional divisors of $\pi:\scrY\to\scrX$ are uniruled, so the same must be true of $X$.
\end{proof}

%------------------------------------------------------------------------------------------

%------------------------------------------------------------------------------------------

\begin{proof}[Proof of Theorem \ref{filling holes}]  Obviously we may assume the monodromy of $R^2f_*\Q_{\scrX^*}$ is trivial.  Let $f:\scrX\to\Delta$ be a flat projective family restricting to the base change of $\scrX^*$ over $\Delta^*$; we may assume the special fiber has no multiple component.  By running the MMP as in the first part of \cite[Theorem 2.1]{KLSV18}, we may assume $\scrX$ has terminal singularities, and so by Proposition~\ref{proposition sympl comp} and Theorem \ref{thm mmp filling} we may assume the special fiber $X$ is a $K$-trivial variety with canonical singularities.  By the proposition again and Theorem~\ref{theorem basic symplectic}, $X$ is symplectic.  Take a $\Q$-factorial terminalization $\pi:Y \to X$ and consider the diagram \eqref{eq defo diag} for $\pi$. With the notations used there, the deformation $\scrY \to \Def(Y)$ is locally trivial by \cite[Main Theorem]{Nam06}. By \cite[Theorem~1]{Nam06} the induced map $p:\Def(Y) \to \Def(X)$ is finite and surjective. Thus, the classifying map $\Delta\to \Def(X)$ of the family $\scrX\to\Delta$ can be lifted to $\Def(Y)$ over a finite cover $\Delta' \to \Delta$. The pullback $\scrY_{\Delta'}$ is then the claimed family; it only remains to show that it is isomorphic to the pullback $\scrX_{\Delta'}$ outside the central fiber.  This is because for $t'\in \Delta'^*$ mapping to $t\in \Delta^*$ we have that $\mathscr{Y}_{t'} \to \scrX_t$ is a proper birational morphism between $\Q$-factorial terminal $K$-trivial varieties and thus is an isomorphism.
\end{proof}

\begin{remark}
The techniques of \cite{KLSV18} are used to \emph{fill in} varieties over projective period points in the interior of the period domain. We would like to point out that in the smooth case this technique of ``filling holes'' has been used independently by Odaka--Oshima for a different purpose, see the second paragraph in the first proof of Claim~8.10 of \cite{OO18}.
\end{remark}

%------------------------------------------------------------------------------------------
%------------------------------------------------------------------------------------------
%------------------------------------------------------------------------------------------
\section{Monodromy and Torelli theorems}\label{section torelli}
%------------------------------------------------------------------------------------------
%------------------------------------------------------------------------------------------
%------------------------------------------------------------------------------------------
Fix a lattice $\Lambda$ and denote its quadratic form by $q$.  
\begin{definition}
We say that a Hodge structure on $\Lambda$ is \emph{semi-polarized} (by the form $q$) if $q:\Lambda\otimes\Lambda\to\Z(-2)$ is a morphism of Hodge structures.  We furthermore say a semi-polarized Hodge structure is \emph{hyperk\"ahler} if it is pure of weight two with $h^{2,0}=h^{0,2}=1$, the signature of $q$ is $(3,b_2-3)$, and $q$ is positive-definite on the real space underlying $H^{2,0}\oplus H^{0,2}$.  Hyperk\"ahler Hodge structures on $\Lambda$ are parametrized by the period domain
\[\Omega_\Lambda:=\{[\sigma]\in \P(\Lambda_\C)\mid q(\sigma)=0, q(\sigma,\bar\sigma)>0\}.  \]
\end{definition}

 Let $X^+$ be a primitive symplectic variety with $(H^2(X^+,\Z)_\tf,q_{X^+})\cong (\Lambda,q)$, and let $\gothM^+$ be the moduli space of $\Lambda$-marked locally trivial deformations of $X^+$.  Note that $\gothM^+$ is a union of connected components of the full moduli space $\gothM_\Lambda$ of $\Lambda$-marked primitive symplectic varieties from Section \ref{subsection inseparable}.

Set $\Omega:=\Omega_\Lambda$.  We have a period map $P:\mathfrak{M}^+\to\Omega$ which is a local isomorphism by the local Torelli theorem (Proposition \ref{proposition local torelli}).  Furthermore, inseparable points of $\mathfrak{M}^+$ lie above proper Mumford--Tate subdomains of $\Omega$ by Corollary \ref{cor bij mt general}, so as in\footnote{Huybrechts uses that the inseparability only occurs above Noether--Lefschetz loci, but the same argument works for any countable union of proper complex analytic subvarieties.} \cite[Corollary 4.10]{Huy12} we have a factorization
\[\xymatrix{
&\overline{\mathfrak{M}}^+\ar[rd]^{\xxoverline{P}}&\\
\mathfrak{M}^+\ar[ru]^{H}\ar[rr]_{P}&&\Omega
}\]
where $H$ is the Hausdorff reduction of $\mathfrak{M}^+$ and $\xxoverline{P}$ is a local homeomorphism.  For each $x\in \overline{\mathfrak{M}}^+$, a local basis is provided by images $H(B)$ of open balls $x\in B\subset\mathfrak{M}^+$ over which there is a universal family for $x$.

Note that $\O(\Lambda)$ acts on each of $\mathfrak{M}^+$, $\overline{\mathfrak{M}}^+$, and $\Omega$ by changing the marking, and the three maps $H,P,\xxoverline{P}$ respect these actions.  For any connected component $\gothM$ of $\gothM^+$, we define $\Mon(\mathfrak{M})\subset O(\Lambda)$ to be the image of the monodromy representation on second cohomology, which is defined up to conjugation.  

The goal of this section is to show:
%------------------------------------------------------------------------------------------
\begin{theorem}\label{theorem torelli}Assume $\rk(\Lambda)\geq 5$ and let $\gothM$ be a connected component of $\gothM^+$.  
\begin{enumerate}
\item The monodromy group $\Mon(\gothM)\subset \O(\Lambda)$ is of finite index;
\item $\xxoverline{P}$ is an isomorphism of $\overline{\mathfrak{M}}$ onto the complement in $\Omega$ of countably many maximal Picard rank periods;
\item If $X^+$ is $\Q$-factorial and terminal, then the same is true of every point $(X,\mu)\in\mathfrak{M}$ and $\xxoverline{P}$ is an isomorphism of $\overline{\mathfrak{M}}$ onto $\Omega$.
\end{enumerate}
\end{theorem}

Theorem \ref{theorem torelli} immediately yields parts (1), (3), and (4) of Theorem \ref{maintorelli}.  Before the proof, we briefly recall the classification of orbit closures in $\Omega$ under an arithmetic lattice, which is crucial to the argument.

%------------------------------------------------------------------------------------------
\subsection{Reminder on orbit closures}\label{subsection orbits}
%------------------------------------------------------------------------------------------
\begin{definition}\label{definition rational rank}
The \emph{rational rank} of a hyperk\"ahler period $p\in\Omega$ is defined as $$\rrk(p):=\dim_\Q \left(\left(H^{2,0}\oplus H^{0,2}\right) \cap \Lambda_\Q\right) \in \{0,1,2\}.$$
We define the rational rank of a primitive symplectic variety to be the rational rank of its Hodge structure on second cohomology.
\end{definition}

Recall that the period domain $\Omega$ can be thought of as the oriented positive Grassmannian $\Gr^{++}(2,\Lambda_\R)$.  For a rational positive-definite sublattice $\ell\subset \Lambda_\Q$ with $\rk(\ell)\leq 2$, we define $T_\ell$ to be the locus of periods for which $\ell\subset(H^{2,0}\oplus H^{0,2})_\R$.  Obviously $T_\ell\supset T_{\ell'}$ if $\ell\subset\ell'$.  Note that if $\rk(\ell)=2$, then $T_\ell$ is a pair of conjugate maximal Picard rank points (and all such pairs arise this way).  For $\rk(\ell)=1$, the set $T_\ell$ is isomorphic to the space $S^{+}(\ell^\perp_\R)$ positive unit-norm vectors in $\ell^\perp_\R$, which is a totally real submanifold of $\Omega$ of real dimension $\rk(\Lambda)-2$.

The important point is that orbit closures for the action of a finite index subgroup $\Gamma\subset\O(\Lambda)$ on the period domain $\Omega$ are classified according to rational rank.  
\begin{proposition}[Theorem 4.8 of \cite{Ver15} and Theorem 2.5 of \cite{Ver17}]  \label{proposition orbits}Assume $\rk(\Lambda)\geq 5$.  We have for $p\in\Omega$:
\begin{enumerate}
\item If $\rrk(p)=0$, then $\overline{\Gamma\cdot p}=\Omega$;
\item If $\rrk(p)=1$, then $\overline{\Gamma\cdot p}$ is a (countable) union of $T_\ell$ with $\rk(\ell)=1$;
\item If $\rrk(p)=2$, then $\overline{\Gamma\cdot p}$ is a (countable) union of $T_\ell$ with $\rk(\ell)=2$.
\end{enumerate} 
\end{proposition}

%------------------------------------------------------------------------------------------
\subsection{Proof of Theorem \ref{theorem torelli}}
%------------------------------------------------------------------------------------------

We divide the proof into five steps.  Parts (1), (2), and (3) are proven in steps 4, 5(a), and 5(b), respectively.
\vskip1em\noindent
\emph{Step 1.}  Let $p\in \Omega$ be a very general period with Picard group generated by a positive vector.  Then $\xxoverline{P}^{-1}(p)$ is finite.
\begin{proof}  In fact, its equivalent to show $P^{-1}(p)$ is finite by the assumption on the Picard rank.  For the following lemma, we say an ample line bundle $L$ on a primitive symplectic variety $X$ has BBF square $d$ if $q_X(\mathrm{c}_1(L))=d$.
\begin{proposition}  Pairs $(X,L)$ consisting of a primitive symplectic variety $X$ of a fixed locally trivial deformation type and an ample line bundle $L$ with fixed BBF square form a bounded family.
\end{proposition}
\begin{proof}  Using that the Fujiki constants are locally trivially deformation-invariant and \cite[Theorem 2.4]{Matsusaka}, for any such pair $(X,L)$, the variety $X$ can be embedded with bounded degree in $\P^N$ for some fixed $N$ via the sections of some fixed power $L^k$.  Let $H$ be the corresponding Hilbert scheme of subschemes of $\P^N$ of bounded degree, and let $f:\mathscr{X}\to H$ be the universal family. Let $H' \subset H$ denote the subset over which the fibers of $f$ are primitive symplectic. By semi-continuity and openness of symplecticity, $H'\subset H$ is open.

\begin{lemma}There is a stratification of $H'$ by locally closed reduced subschemes over which $\mathscr{X}$ is locally trivial.
\end{lemma}
\begin{proof}There is a stratification $H_i$ of $H'$ by locally closed reduced subschemes along which the second Betti numbers $(R^2f_*\Q_\scrX)_t$ are constant, for instance by using \'etale cohomology.  By Corollary \ref{corollary betti locus}, $\mathscr{X}$ is locally trivial in an analytic neighborhood of every point $t$ in each $H_i$, and so $\mathscr{X}$ is locally trivial on each $H_i$.
\end{proof}

  It follows from the lemma that the set of pairs $(X,L)$ as in the statement of the proposition together with a choice of an embedding into $\P^N$ as above is a locally closed subscheme $U$ of $H$.  The $\C$-points of the quotient stack $[\PGL_{N+1}\backslash U]$ then parametrize isomorphism classes of the pairs $(X,L)$.  The $\PGL_{N+1}$ action has finite stabilizers on $U$ by Lemma \ref{lemma existence universal deformation}, so by general theory $[\PGL_{N+1}\backslash U]$ is a Deligne--Mumford stack and there is a finite-type \'etale atlas $S\rightarrow [\PGL_{N+1}\backslash U]$.  

To summarize, there is (depending on the fixed locally trivial deformation type and the fixed BBF square) a finite-type scheme $S$ and a locally trivial family $\mathscr{X}\to S$ of primitive symplectic varieties and a relatively ample $\mathscr{L}$ on $\mathscr{X}$ which has the property that every $(X,L)$ as in the statement of the lemma appears finitely many times (and at least once) as a fiber.
\end{proof}
Each component $S_0$ of the scheme $S$ constructed in the proof of the lemma has a period map of the form $P_v:S_0\rightarrow \O(v^\perp)\backslash\Omega_{v^\perp}$ for some $v\in \Lambda$ with fixed square $q(v)=d$, where we think of $\Omega_{v^\perp}=\P(v^\perp)\cap\Omega$.  Moreover, $P_v$ is a local isomorphism and therefore quasifinite, as by e.g. \cite[Theorem 3.10]{B72} the fibers are algebraic.  

Now, for $p\in \Omega$ as in the original claim, suppose $q(v)=d$ for a generator $v$ of the Picard group.  It follows that there are finitely many isomorphism classes of pairs $(X,L)$ where $X$ is a primitive symplectic variety that is locally trivially deformation-equivalent to $X^+$, and $L$ is an ample bundle of BBF square $d$, and the primitive parts of $H^2(X,\Z)_\tf$ and $p$ are abstractly isomorphic as polarized Hodge structures. By the assumption on the Picard rank, there are then finitely many isomorphism classes of projective $X$ locally-trivially deformation equivalent to $X^+$ and with $H^2(X,\Z)_\tf$ abstractly isomorphic to $p$ as semi-polarized Hodge structures.  Moreover, $\Aut(p)=\pm 1$, so for each such $X$ there are finitely many such isomorphisms.

To finish, by Theorem \ref{theorem projectivity criterion} every point in $P^{-1}(p)$ is projective and uniquely polarized by a class of BBF square $d$, and the claim follows.
\end{proof}
\vskip1em\noindent

\hspace{\parindent}For the next step, let $\Omega_{\rrk=0}\subset \Omega$ be the rational-rank-zero locus, let $\overline{\mathfrak{M}}^+_{\rrk=0}\subset \overline{\gothM}^+$ be the preimage of $\Omega_{\rrk=0}$ under $\xxoverline{P}$, and let $\xxoverline{P}_{\rrk=0}$ be the restriction of $\xxoverline{P}$ to $\overline{\mathfrak{M}}^+_{\rrk=0}$.  Note that since we are assuming $\rk(\Lambda)\geq 5$, every $p\in\Omega_{\rrk=0}$ has dense $\O(\Lambda)$-orbit by Proposition \ref{proposition orbits}.
\vskip1em\noindent
\emph{Step 2.} The map $\xxoverline{P}_{\rrk=0}$ is a covering map onto $\Omega_{\rrk=0}$. 
\begin{proof}The claim follows from the following two lemmas.
\begin{lemma}\label{lemma pbar finite}
The map $\xxoverline{P}_{\rrk=0}$ has finite fibers of constant size.  In particular,  it is surjective onto $\Omega_{\rrk=0}$.
\end{lemma}
\begin{proof}By the previous step there is a point $p_0\in \Omega_{\rrk=0}$ over which $\xxoverline{P}^{-1}(p_0)$ is finite of size $N$, and therefore $\xxoverline{P}^{-1}(p)$ is finite of size $\leq N$ for every point $p\in \Omega_{\rrk=0}$.  Indeed, if some $p\in \Omega_{\rrk=0}$ had at least $N+1$ preimages then by Hausdorffness we can find pairwise non-intersecting open neighborhoods around any $N+1$ points in the fiber $\xxoverline{P}^{-1}(p)$ that map isomorphically to the same open neighborhood $V$ of $p$, but $p_0$ has dense orbit.  Interchanging $p_0$ and $p$, we see that in fact the fibers are finite of constant size.  
\end{proof}
\begin{lemma}\label{lemma covering space}Suppose $f:X\to Y$ is a local homeomorphism between two Hausdorff\footnote{In fact, only Hausdorffness on the source is used.} topological spaces.  If $f$ has finite fibers of constant size, then it is a covering map onto its image.
\end{lemma}
\begin{proof}
For any $y\in Y$, because $f^{-1}(y)$ is finite we may find nonintersecting open sets $U_x$ around each point $x\in f^{-1}(y)$ on which $f$ is a homeomorphism, and by shrinking we may further assume all the $U_x$ have the same image $U$.  It follows from the assumption on fiber size that $f^{-1}(U)=\bigcup_{x \in f^{-1}(y)} U_x$.
\end{proof}

\end{proof}

\vskip1em\noindent
%------------------------------------------------------------------------------------------
\emph{Step 3.} The map $\xxoverline{P}_{\rrk=0}$ is an isomorphism of $\overline{\mathfrak{M}}_{\rrk=0}$ onto $\Omega_{\rrk=0}$.

\begin{proof}
The rational-rank-zero locus is 
$$\Omega_{\rrk=0}:=\Omega\smallsetminus \bigcup_{\ell\neq 0} T_\ell$$
 in the notation of Section~\ref{subsection orbits}, and each $T_\ell$ is a closed submanifold of real codimension $\rk(\Lambda)-2$.  Assuming $\rk(\Lambda)\geq 5$, we have that $\Omega_{\rrk=0}$ is locally path-connected and path-connected by \cite[Lemma 4.10]{Ver13} and moreover locally simply connected and simply connected by the following lemma, as the same is true of $\Omega$.
\begin{lemma}\label{lemma simply connected}
If $M$ is a simply connected smooth manifold and $S$ is a countable union of closed submanifolds of (real) codimension $\geq 3$, then $M\smallsetminus S$ is simply connected.
\end{lemma}
\begin{proof}
This argument is taken from a MathOverflow answer of Martin M. W. \cite{MO}. The result is well-known when $S$ is a single closed submanifold of codimension $\geq 3$. The space of nulhomotopies $S^1\times [0,1]\to M$ of a given path with the compact open topology is a Baire space and the set of homotopies avoiding a single closed submanifold of codimension $\geq 3$ is a dense open subset.  Therefore, the set of homotopies avoiding $S$ is nonempty (and in fact dense) by definition of a Baire space.
\end{proof}

  Thus, the claim follows from the previous step.
\end{proof}

\vskip1em\noindent
%------------------------------------------------------------------------------------------
\emph{Step 4.} The subgroup $\Mon(\mathfrak{M})$ has finite index in $\O(\Lambda)$.

\begin{proof} The map	$\xxoverline{P}:\overline{\mathfrak{M}}^+_{\rrk=0}\to\Omega_{\rrk=0}$ has finite degree and $\Omega_{\rrk=0}$ is path-connected.  Therefore, $\mathfrak{M}^+$ has finitely many connected components. The group $\Mon(\mathfrak{M})$ is the stabilizer of the component $\gothM$, and is therefore finite index.
\end{proof}

\vskip1em\noindent
%------------------------------------------------------------------------------------------
\emph{Step 5(a).} The map $\xxoverline{P}$ is an isomorphism of $\overline{\mathfrak{M}}$ onto the complement in $\Omega$ of countably many maximal Picard rank periods.

\begin{proof} By step 3, it is enough to show that the image of $\gothM$ under $P$ contains the locus $\Omega_{\rrk\leq 1}$ of non-maximal Picard rank periods.  The image is open and $\Mon(\gothM)$-invariant, whereas by Proposition \ref{proposition orbits} and the previous step a $\Mon(\gothM)$ orbit closure in $\Omega$ must be a union of $T_\ell$ or all of $\Omega$.  It is therefore enough to show that for any rank one sublattice $\ell\subset \Lambda$, a very general point of $T_\ell$ is contained in $P(\gothM)$.  

Considering a projective $(X,\mu)\in \gothM$ with a polarization $v$ that is orthogonal to $\ell$, we obtain a period map $P_v:S_0\rightarrow \O(v^\perp)\backslash\Omega_{v^\perp}$ as in Step~1 corresponding to a family of locally trivial deformations of $X$ over $S_0$.  The complement of $P_v(S_0)$ is a locally closed subvariety of $\O(v^\perp)\backslash\Omega_{v^\perp}$ and its preimage $V$ in $\Omega_{v^\perp}$ is therefore also a locally closed analytic subvariety.  

It suffices to show that $T_\ell\cap \Omega_{v^\perp}$ is not contained in $V$.  But $T_\ell$ is totally real and has half the (real) dimension of $\Omega$, so the tangent space to $T_\ell\cap\Omega_{v^\perp}$ at a point $p$ is not contained in any proper complex subspace of $T_p\Omega_{v^\perp}$.  It follows that if $T_\ell$ were contained in $V$, it must be contained in the singular locus of $V$, and so by induction we get a contradiction.

\end{proof}

\vskip1em\noindent
%------------------------------------------------------------------------------------------
\emph{Step 5(b).} When $X^+$ is $\Q$-factorial and terminal, then the same is true of every point $(X,\mu)\in\mathfrak{M}$ and $\xxoverline{P}$ is an isomorphism of $\overline{\mathfrak{M}}$ onto $\Omega$.

\begin{proof} The first claim follows from Lemma \ref{lemma deformation qfactorial}.  For the second claim, by the previous step, it remains to show $P(\gothM)$ contains all maximal Picard rank points, which are in particular projective by Theorem \ref{theorem projectivity criterion}. 

Now for any maximal Picard rank period $p$, let $v\in\Lambda$ be a positive vector which is Hodge with respect to $p$.  A very general deformation of $p$ for which $v$ remains algebraic is in the image of $P$, and the period map $P_v:S_0\rightarrow \O(v^\perp)\backslash\Omega_{v^\perp}$ from step 1 is dominant, so we can find a curve $B\subset \O(v^\perp)\backslash\Omega_{v^\perp}$ through $p$ such that an open set $U\subset B$ lifts to $S_0$, possibly after a base change.  Now apply\footnote{Or \cite[Theorem 1.7]{KLSV18} in the smooth case.} Proposition \ref{filling holes}.
\end{proof}
%------------------------------------------------------------------------------------------

This concludes the proof.

\begin{remark}
The argument given in Step 1 of the proof of the theorem together with Huybrechts' surjectivity of the period map \cite[Theorem~8.1]{Huy99} implies that the monodromy group is of finite index for irreducible symplectic manifolds even when $b_2=4$ (for $b_2=3$ it is automatic).  For an argument not using Huybrechts' theorem, see \cite[Theorem 2.6]{Ver19}.
\end{remark}

\begin{remark}Some ideas similar to those appearing in the proof of Theorem \ref{theorem torelli} have also been used recently by Huybrechts \cite{Huy18} to prove some finiteness results for hyperk\"ahler manifolds, and these arguments can likely be adapted to the singular setting.
\end{remark}

%------------------------------------------------------------------------------------------
\section{\texorpdfstring{$\Q$}{Q}-factorial terminalizations}\label{section qfactorial terminalizations}
%------------------------------------------------------------------------------------------

If $X$ is an algebraic variety, then by \cite[Corollary 1.4.3]{BCHM10} there exists $\Q$-factorial terminalization $\pi:Y\to X$. This is often crucial in the theory of singular symplectic varieties. On the other hand, even if you are mainly interested in projective symplectic varieties, it is often necessary to consider also compact Kähler varieties and certainly the methods of \cite{BCHM10} are not yet established in the Kähler case. The main result of this section, Theorem \ref{theorem qfactorial terminalization}, partially remedies this in the case of primitive symplectic varieties. If we start with a primitive symplectic variety with second Betti number $\geq 5$, it establishes the existence of $\Q$-factorial terminalizations on a bimeromorphic model which is locally trivially deformation equivalent to the initial variety.

In fact, by Theorem \ref{theorem inseparable implies birational} the following is slightly stronger, though we expect it to be equivalent.  For a normal variety $X$ we denote by $\omega_X$ the push forward of the canonical bundle along the inclusion of the regular locus.

%------------------------------------------------------------------------------------------
\begin{theorem}\label{theorem qfactorial terminalization}
Let $X$ be a primitive symplectic variety satisfying $b_2(X)\geq 5$. Then there exist a primitive symplectic variety $X'$ which is inseparable from $X$ in (locally trivial) moduli and a $\Q$-factorial terminalization of $X'$, that is, a proper bimeromorphic morphism $\pi:Y\to X'$ such that $Y$ has only $\Q$-factorial terminal singularities and $\pi^*\omega_{X'}=\omega_Y = \sO_Y$. In particular, $Y$ is a primitive symplectic variety.
\end{theorem}
%------------------------------------------------------------------------------------------
As a consequence of the fact that bimeromorphic varieties without compact curves are isomorphic, see e.g. the proof of Corollary~\ref{cor bij mt general}, we obtain:
\begin{corollary}
Let $X$ be as in Theorem \ref{theorem qfactorial terminalization}, and additionally assume it has Picard rank zero. Then $X$ has a $\Q$-factorial terminalization.\qed
\end{corollary}

The proof of Theorem \ref{theorem qfactorial terminalization} is obtained by combining Proposition \ref{prop defo} with Corollary \ref{corollary density of projectives2}, Theorems \ref{theorem inseparable implies birational} and \ref{theorem torelli}, and the existence of $\Q$-factorial terminalizations of projective varieties.

\begin{proof}[Proof of Theorem \ref{theorem qfactorial terminalization}]
Let us consider the universal locally trivial deformation $\scrX\to\Def^\lt(X)$ and choose $t\in \Def^\lt(X)$ nearby such that $X_0:=\scrX_t$ is projective. Take a $\Q$-factorial terminalization $Y_0\to X_0$, denote $N$ the $q_{Y_0}$-orthogonal complement of $H^2(X_0,\Q)$ in $H^2(Y_0,\Q)$, and consider the universal deformation of the pair
\[
\xymatrix{
\scrY_0 \ar[r]\ar[d]& \ar[d]\scrX_0 \\
\Def(Y_0,N) \ar[r]^p& \Def^\lt(X_0)
}
\]
given by Proposition~\ref{prop defo}. By Lemma \ref{lemma deformation qfactorial}, we may assume that every fiber of the map $\scrY_0\to \Def(Y_0,N)$ is $\Q$-factorial and by local triviality and Theorem~\ref{theorem basic symplectic}, every fiber has terminal singularities. In other words, for all $s \in \Def(Y_0,N)$ the morphism $(\scrY_0)_s \to (\scrX_0)_{p(s)}$ is a $\Q$-factorial terminalization.

If $\rrk(X) = 0$, then by Proposition \ref{proposition orbits} and Theorem \ref{theorem torelli} there is a point $t' \in \Def^{\lt}(X_0)$ such that the fiber $X':=(\scrX_0)_{t'}$ and $X$ are inseparable in moduli. By construction, $X'$ has a $\Q$-factorial terminalization, and by Theorem \ref{theorem inseparable implies birational}, $X'$ is bimeromorphic to $X$.  If $\rrk(X)=1$, projective periods are still dense in the orbit closure of the period of $X$ by Theorem~\ref{theorem projectivity criterion}, so the same argument can be applied by choosing the period of $X_0$ to be in the orbit closure of the period of $X$.  Finally, varieties $X$ with $\rrk(X)=2$ are projective so there the result is known anyway by \cite[Corollary~1.4.3]{BCHM10}.
\end{proof}

%------------------------------------------------------------------------------------------
As an application, we can give examples of divisorially $\Q$-factorial but not $\Q$-factorial varieties.

%------------------------------------------------------------------------------------------
\begin{example}\label{example qfactoriality}
Consider a projective irreducible symplectic manifold $Y$ of dimenison $2n$ admitting a small contraction $\pi:Y \to X$ where $X$ is a projective primitive symplectic variety and the exceptional locus of $\pi$ is isomorphic to $\P^n$. As $\pi$ has connected fibers, $\P^n$ must be contracted to a point and thus $X$ has an isolated singularity. 
Such examples can be realized on the Hilbert scheme $Y=S^{[n]}$ of $n$ points on a K3 surface $S$ containing a smooth rational curve. As the contraction is small, the variety $X$ is not $\Q$-factorial. By \cite[Theorem 4.1, Proposition 4.5, and Proposition 5.8]{BL16}, this contraction deforms over a smooth hypersurface in $\Def(Y)$. 

We denote by $\ell \subset \P^n \subset Y$ a line and by $L$ the unique line bundle on $Y$ with $q_Y(\mathrm{c}_1(L),\cdot)= (\ell,\cdot)$ where the right--hand side denotes the pairing $N_1(Y)_\Q \tensor N^1(Y)_\Q \to \Q$. It follows that $\mathrm{c}_1(L)$ is $q$-orthogonal to the pullback of any ample divisor on $X$, hence $q_{Y}(\mathrm{c}_1(L))=q_{Y}(\ell)<0$. Replacing $X$ by a small locally trivial deformation, we may assume:

\begin{enumerate}
\item\label{example qfactoriality item two} The contraction $\pi:Y\to X$ deforms and has $\P^n$ as its exceptional set.
\item\label{example qfactoriality item one} The varieties $X$ and $Y$ are Kähler and non-algebraic such that the Picard group of $X$ is trivial and $\Pic(Y)$ has rank one. 
\item\label{example qfactoriality item four} There are no divisors on $Y$, in particular, $X$ is not $\Q$-factorial but divisorially $\Q$-factorial in the sense of Definition \ref{definition qfactorial}.
\end{enumerate}
For \eqref{example qfactoriality item four}, if $L^n$ were represented by an effective divisor $D$, then since $q_X(D)<0$ it is exceptional \cite[Theorem 4.5]{Bou04} and hence uniruled \cite[Proposition 4.7]{Bou04}.  However, the only curves on $Y$ are the ones contracted by $\pi$.
\end{example}
%------------------------------------------------------------------------------------------

%------------------------------------------------------------------------------------------

\bibliography{literatur}
\bibliographystyle{alpha}

%-------------------------------------
\end{document}